\documentclass[12pt,twoside,british]{extarticle}

\usepackage[T1]{fontenc}
\usepackage[latin9]{inputenc}
\usepackage[a4paper]{geometry}
\usepackage{babel}
\usepackage{amsmath}
\usepackage{amsthm}
\usepackage{amssymb}
\usepackage{stmaryrd}
\usepackage[numbers]{natbib}
\usepackage[unicode=true,pdfusetitle,
 bookmarks=true,bookmarksnumbered=false,bookmarksopen=false,
 breaklinks=false,pdfborder={0 0 0},pdfborderstyle={},backref=false,colorlinks=false]
 {hyperref}
\hypersetup{
 linkcolor=black, citecolor=black, urlcolor=blue, filecolor=blue, pdfpagelayout=OneColumn, pdfnewwindow=true, pdfstartview=XYZ, plainpages=false}
\usepackage{breakurl}

\makeatletter
\numberwithin{equation}{section}
\numberwithin{figure}{section}
\theoremstyle{plain}
\newtheorem{thm}{\protect\theoremname}[section]
  \theoremstyle{plain}
  \newtheorem{lem}[thm]{\protect\lemmaname}
  \theoremstyle{definition}
  \newtheorem{defn}[thm]{\protect\definitionname}
  \theoremstyle{remark}
  \newtheorem{rem}[thm]{\protect\remarkname}

\@ifundefined{date}{}{\date{}}
\usepackage{caption}
\usepackage{graphicx}
\usepackage[nottoc]{tocbibind}

\makeatother

  \providecommand{\definitionname}{Definition}
  \providecommand{\lemmaname}{Lemma}
  \providecommand{\remarkname}{Remark}
\providecommand{\theoremname}{Theorem}

\begin{document}

\title{A stochastic Pontryagin maximum principle on the Sierpinski gasket}

\author{Xuan~Liu\thanks{Nomura International, 30/FL Two International Finance Centre, Hong
Kong. Email: \protect\href{mailto:chamonixliu@163.com}{chamonixliu@163.com}.\protect \\
This research was carried out when the author was reading DPhil in
Mathematics at the University of Oxford.}}
\maketitle
\begin{abstract}
In this paper, we consider stochastic control problems on the Sierpinski
gasket. An order comparison lemma is derived using heat kernel estimate
for Brownian motion on the gasket. Using the order comparison lemma
and techniques of BSDEs, we establish a Pontryagin stochastic maximum
principle for these control problems. It turns out that the stochastic
maximum principle on the Sierpinski gasket involves two necessity
equations in contrast to its counterpart on Euclidean spaces. This
effect is due to singularity between the Hausdorff measure and the
energy dominant measure on the gasket, which is a common feature shared
by many fractal spaces. The linear regulator problems on the gasket
is also considered as an example. 
\end{abstract}

\section{Introduction}

Recently, to study non-linear analysis on the Sierpinski gasket, \citep{LQ16}
developed a theory of backward stochastic differential equations (BSDEs)
on the Sierpinski gasket. BSDEs and related stochastic analysis on
fractals, though initially considered as efficient tools to treat
quasi-linear parabolic PDEs on fractals, also have interests on their
own from a mathematical finance point of view. Several interesting
mathematical finance problems are formulated as stochastic control
problems on Euclidean spaces, which are based upon the assumption
that uncertainties in financial models are sourced from Brownian filtration
on Euclidean spaces. However, it had been widely observed from the
real data that many financial time series exhibit fractal behaviours
(see, for example, \citep{AIV02,Man13,DW14} and etc.), which suggests
the possibility that uncertainties in the markets might come from
filtrations exhibiting fractal structures. Therefore, it is of significance
to consider stochastic control problems for controlled systems with
noise coming from filtrations determined by the diffusions on fractals. 

The motivation of this paper is to establish a stochastic Pontryagin
maximum principle for stochastic control problems on the Sierpinski
gasket, with uncertainties in the controlled dynamic systems generated
by the diffusion on the gasket. It turns out that, in contrast to
its counterpart on Euclidean spaces, the stochastic maximum principle
on the gasket consists of two necessity equations rather than a single
one (see \citep{Pen90} and \citep[Section 3.2]{Yong99}). As we shall
see, this is due to the singularity between two measures which are
both necessary for analysis on fractals. 

This paper is organized as follows. In Section \ref{sec:-1}, we introduce
notations which will be enforced throughout this paper, and review
some related results in literature. The main results of this paper
is formulated and collected in Section \ref{sec:-2}. Section \ref{sec:-3}
is devoted to the proof of the stochastic maximum principle on the
Sierpinski gasket. The linear regulator problem on the gasket is considered
in Section \ref{sec:} as an example. Though results of this paper
are established for two-dimensional Sierpinski gasket, we however
believe that our results also hold for higher-dimensional cases, where
argument in this paper should remain valid. 

\section{\label{sec:-1}Notations and related results}

In this section, we introduce notations which will be enforced throughout
this paper. We also review several results in literature needed in
the following sections.

Let $\mathrm{V}_{0}=\{p_{1},p_{2},p_{3}\}\subseteq\mathbb{R}^{2}$
with $p_{1}=(0,0),\;p_{2}=(1,0),\;p_{3}=(\tfrac{1}{2},\tfrac{\sqrt{3}}{2})$,
and $\mathrm{F}_{i}:\mathbb{R}^{2}\to\mathbb{R}^{2},\;i=1,2,3$ be
the contraction mappings given by $\mathrm{F}_{i}(x)=\tfrac{1}{2}(x+p_{i}),\;x\in\mathbb{R}^{2},\;i=1,2,3.$
Define $\mathrm{V}_{m},\;m\in\mathbb{N}$ inductively by $\mathrm{V}_{m+1}=\mathrm{F}_{1}(\mathrm{V}_{m})\cup\mathrm{F}_{2}(\mathrm{V}_{m})\cup\mathrm{F}_{3}(\mathrm{V}_{m}),\;m\in\mathbb{N}$,
and $\mathrm{V}_{\ast}=\bigcup_{m=0}^{\infty}\mathrm{V}_{m}$. The
(two-dimensional) \emph{Sierpinski gasket} is defined to be the closure
$\mathbb{S}=\bar{\mathrm{V}}_{\ast}$ of $\mathrm{V}_{\ast}$ in $\mathbb{R}^{2}$. 

For a given set $\mathrm{V}$, we denote by $\ell(\mathrm{V})$ the
space of all real-valued functions on $\mathrm{V}$. The \emph{standard
Dirichlet form} $(\mathcal{E},\mathcal{F}(\mathbb{S}))$ on the Sierpinski
gasket $\mathbb{S}$ is defined by
\[
\left\{ \begin{aligned}\mathcal{E}(u,v) & =\lim_{m\to\infty}\mathcal{E}^{(m)}(u,v),\;\;u,v\in\mathcal{F}(\mathbb{S}),\\
\mathcal{F}(\mathbb{S}) & =\Big\{ u\in\ell(\mathbb{S}):\lim_{m\to\infty}\mathcal{E}^{(m)}(u,u)<\infty\Big\},
\end{aligned}
\right.
\]
where the forms $\mathcal{E}^{(m)},\;m\in\mathbb{N}$ are given by
\[
\mathcal{E}^{(m)}(u,v)=\Big(\frac{5}{3}\Big)^{m}\sum_{x,y\in\mathrm{V}_{m}:|x-y|=2^{-m}}[u(x)-u(y)][v(x)-v(y)],\;\;u,v\in\ell(\mathrm{V}_{m}).
\]

Let $\nu$ be the \emph{Hausdorff measure} on $\mathbb{S}$ with weight
$(\frac{1}{3},\frac{1}{3},\frac{1}{3})$, that is, $\nu$ is the unique
Borel probability measure on $\mathbb{S}$ such that $\nu(\mathbb{S}_{[\omega]_{m}})=3^{-m}$
for each $\omega\in\tilde{\Omega}$ and each $m\in\mathbb{N}$. Then
the form $\mathcal{E}$ is a regular Dirichlet form on $L^{2}(\mathbb{S};\nu)$,
and $\mathcal{F}(\mathbb{S})$ is the corresponding Dirichlet space. 

The \emph{Kusuoka measure} $\mu$ on $\mathbb{S}$ is defined by $\mu=(\mu_{1}+\mu_{2}+\mu_{3})/3$,
where $\mu_{i}$ is the energy measures of the harmonic function with
boundary value $1_{p_{i}}$, which is the unique minimizer of $\inf\{\mathcal{E}(u,u):u\in\mathcal{F}(\mathbb{S})\allowbreak\;\text{and}\allowbreak\;u|_{\mathrm{V}_{0}}=1_{p_{i}}\}$.

According to the general theory of Dirichlet forms and Markov processes
(see \citep[Chapter 7]{FOT10}), associated to the form $(\mathcal{E},\mathcal{F}(\mathbb{S}))$
there exists a standard Hunt process $\mathbf{M}=\big(\Omega,\mathcal{F},\{X_{t}\}_{t\in[0,\infty]},\allowbreak\{\mathbb{P}_{x}\}_{x\in\mathbb{S}\cup\{\Delta\}}\big)$
with state space $\mathbb{S}$, where $\Delta$ is the ``cemetery''
of $\mathbf{M}$. The process $\{X_{t}\}_{t\ge0}$ is called \emph{Brownian
motion }on $\mathbb{S}$. The semigroup of $\{X_{t}\}_{t\ge0}$ will
be denoted by $\{P_{t}\}_{t\ge0}$.

Let $\mathcal{P}(\mathbb{S})$ be the family of all Borel probability
measures on $\mathbb{S}$. For each $\lambda\in\mathcal{P}(\mathbb{S})$,
the probability measure $\mathbb{P}_{\lambda}$ on $\Omega$ is defined
by $\mathbb{P}_{\lambda}(A)=\int_{\mathbb{S}}\mathbb{P}_{x}(A)\lambda(dx),\;A\in\mathcal{F}$.
The expectation with respect to $\mathbb{P}_{\lambda}$ will be denoted
by $\mathbb{E}_{\lambda}$. Let $\mathcal{F}_{t}^{0}=\sigma\left(X_{r}:r\le t\right),\;t\ge0$,
$\mathcal{F}_{t}^{\lambda}$ the $\mathbb{P}_{\lambda}$-completion
of $\mathcal{F}_{t}^{0}$ in $\mathcal{F}$, and $\{\mathcal{F}_{t}\}_{t\ge0}$
the minimal completed admissible filtration (cf. \citep[p. 385]{FOT10})
of $\{X_{t}\}_{t\ge0}$, that is, $\mathcal{F}_{t}=\bigcap_{\lambda\in\mathcal{P}(\mathbb{S})}\mathcal{F}_{t}^{\lambda},\;t\ge0$.

We end this section with a review on the representing martingale on
the Sierpinski gasket. The following result was first shown in \citep[Theorem (5.4)]{Ku89}
(see also \citep[Theorem 2.6]{LQ16}).
\begin{thm}
\label{thm:-}There exists a martingale additive functional $W_{t}$
satisfying the following:

(i) $W_{t}$ has $\mu$ as its energy measure; 

(ii) For any $u\in\mathcal{F}(\mathbb{S})$, there exists a unique
$\zeta\in L^{2}(\mathbb{S};\mu)$ such that 
\begin{equation}
M_{t}^{[u]}=\int_{0}^{t}\zeta(X_{r})dW_{r},\;\;\text{for all}\ t\ge0,\label{eq:-21}
\end{equation}
where $M^{[u]}$ is the martingale part of $u(X_{t})-u(X_{0})$.
\end{thm}
The martingale additive functional $W$ given by (\ref{eq:-21}) is
called the \emph{Brownian martingale} on $\mathbb{S}$. The following
result on the singularity between the Lebesgue-Stieltjes measure induced
by $t\mapsto\langle W\rangle_{t}$ and the Lebesgue measure on $[0,\infty)$
was proved in \citep[Lemma 4.10]{LQ16}.
\begin{lem}
\label{lem:-33}The Lebesgue-Stieltjes measure $d\langle W\rangle_{t}(\omega)$
is singular to the Lebesgue measure $dt$ on $[0,\infty)$ $\mathbb{P}_{\nu}\mbox{-a.e.}\;\omega\in\Omega$.
\end{lem}
The following lemma, which is shown in \citep[Lemma 4.11]{LQ16},
gives the exponential integrability of $\langle W\rangle_{t}$. 
\begin{lem}
\label{lem:-3}For each $f\in L_{+}^{1}(\mu)$ and $\kappa,t>0$,
\[
\sup_{x\in\mathbb{S}}\mathbb{E}_{x}\big(f(X_{t})e^{\kappa\langle W\rangle_{t}}\big)\le\max\{1,t^{-d_{s}/2}\}\Vert f\Vert_{L^{1}(\mu)}\mathrm{E}_{\gamma_{s},\gamma_{s}}[C_{\ast}\kappa\max\{t,t^{\gamma_{s}}\}],
\]
where $C_{\ast}>0$ is a universal constant.
\end{lem}

\section{\label{sec:-2}Formulation of the main result}

Let $\lambda\in\mathcal{P}(\mathbb{S})$ satisfy $\lambda\ll\nu$.
Let the \emph{decision space} $(\mathbb{U},\rho)$ be a separable
metric space. Let $h:\mathbb{R}\to\mathbb{R},\;\allowbreak f_{1}:[0,T]\times\mathbb{R}\times\mathbb{U}\to\mathbb{R},\;\allowbreak f_{2}:[0,T]\times\mathbb{R}\times\mathbb{U}\to\mathbb{R}$
be Borel measurable functions. For any $\mathbb{U}\text{-}$valued
progressively measurable process $u(t)$, we introduce the \emph{cost
functional }
\begin{equation}
J(u)\triangleq\mathbb{E}_{\lambda}\bigg(h(x(T))+\int_{0}^{T}f_{1}(t,x(t),u(t))dt+\int_{0}^{T}f_{2}(t,x(t),u(t))d\langle W\rangle_{t}\bigg),\label{eq:-5}
\end{equation}
for the \emph{controlled system} $x(t)$ of which the dynamics is
given by the following SDE on $\big(\Omega,\mathcal{F},\{\mathcal{F}_{t}^{\lambda}\}_{t\ge0},\mathbb{P}_{\lambda}\big)$:
\begin{equation}
\left\{ \begin{aligned}dx(t) & =b_{1}(t,x(t),u(t))dt+b_{2}(t,x(t),u(t))d\langle W\rangle_{t}\\
 & \quad+\sigma(t,x(t),u(t))dW_{t},\quad t\in(0,T],\;\mathbb{P}_{\lambda}\text{-a.s.},\\
x(0) & =x_{0},
\end{aligned}
\right.\label{eq:-7}
\end{equation}
where $\varphi:[0,T]\times\mathbb{R}\times\mathbb{U}\to\mathbb{R},\;\varphi=b_{1},b_{2},\sigma$
are Borel measurable functions, and $x_{0}\in\mathcal{F}_{0}^{\lambda}$.\footnote{The existence and uniqueness of solutions to (\ref{eq:-7}) can be
easily shown by an a priori estimate similar to \citep[eqn. (3.8), p. 8]{LQ16}.}
\begin{defn}
Denote by $\mathcal{A}[0,T]$ the family of all $\mathbb{U}\text{-}$valued
processes $u(t)$ such that
\begin{equation}
\mathbb{E}_{\lambda}\bigg(|h(x(T))|+\int_{0}^{T}|f_{1}(t,x(t),u(t))|dt+\int_{0}^{T}|f_{2}(t,x(t),u(t))|d\langle W\rangle_{t}\bigg)<\infty,\label{eq:-6}
\end{equation}
where $x(t)$ is the controlled process given by (\ref{eq:-7}). Any
$u\in\mathcal{A}[0,T]$ is called an \emph{admissible control}, and
$(x(\cdot),u(\cdot))$ is called an \emph{admissible pair}.
\end{defn}
We consider the following optimization problem
\begin{equation}
\mathop{\text{minimize}}_{u\in\mathcal{A}[0,T]}\;J(u),\tag{P}\label{eq:-8}
\end{equation}
subject to the controlled dynamics (\ref{eq:-7}). To formulate our
result, we shall need the following definition. 
\begin{defn}
\label{def:-1}We define the measure $\mathfrak{M}_{1}$ on $[0,\infty)\times\Omega$
to be
\begin{equation}
\mathfrak{M}_{1}=dt\times\mathbb{P}_{\lambda},\label{eq:-36}
\end{equation}
and the measure $\mathfrak{M}_{2}$ to be the unique measure on the
optional $\sigma\text{-}$field\footnote{That is, the $\sigma\text{-}$field on $[0,\infty)\times\Omega$ generated
by the family of all right continuous left limit processes.} on $[0,\infty)\times\Omega$ such that
\begin{equation}
\mathfrak{M}_{2}\big(\llbracket\sigma_{1},\sigma_{2}\rrparenthesis\big)=\mathbb{E}_{\lambda}\big(\langle W\rangle_{\sigma_{2}}-\langle W\rangle_{\sigma_{1}}\big),\label{eq:-37}
\end{equation}
for any $\{\mathcal{F}_{t}\}\text{-}$stopping times $\sigma_{1},\sigma_{2}$
with $\sigma_{1}\le\sigma_{2}$, where $\llbracket\sigma_{1},\sigma_{2}\rrparenthesis=\{(t,\omega)\in[0,\infty)\times\Omega:\sigma_{1}(\omega)\le t<\sigma_{2}(\omega)\}$.
\end{defn}
\begin{rem}
\label{rem:}By $\lambda\ll\nu$ and Lemma \ref{lem:-33}, the measures
$\mathfrak{M}_{1}$ and $\mathfrak{M}_{2}$ are mutually singular.
\end{rem}
\begin{thm}
\label{thm:}Let $\lambda\in\mathcal{P}(\mathbb{S})$ be absolutely
continuous with respect to $\nu$. Assume that:\\
$\hypertarget{A.1}{\text{(A.1)}}$
\[
\left\{ \begin{aligned}|\varphi(t,x,u)-\varphi(t,\hat{x},\hat{u})| & \le M|x-\hat{x}|+\rho(u,\hat{u}), & t & \in[0,T],\;x,\hat{x}\in\mathbb{R},\;u,\hat{u}\in\mathbb{U},\\
|\varphi(t,0,u)| & \le M, & t & \in[0,T],\;u\in\mathbb{U},
\end{aligned}
\right.
\]
for $\varphi=b_{1},b_{2},\sigma,f_{1},f_{2},h$, and\\
$\hypertarget{A.2}{\text{(A.2)}}$
\[
\begin{aligned}|\partial_{x}\varphi(t,x,u) & -\partial_{x}\varphi(t,\hat{x},\hat{u})|+|\partial_{x}^{2}\varphi(t,x,u)-\partial_{x}^{2}\varphi(t,\hat{x},\hat{u})|\\
 & \le M|x-\hat{x}|+\rho(u,\hat{u}),\;\;t\in[0,T],\;x,\hat{x}\in\mathbb{R},\;u,\hat{u}\in\mathbb{U},
\end{aligned}
\]
for $\varphi=b_{1},b_{2},\sigma,f_{1},f_{2},h$, where $M>0$ is a
constant.

Suppose that $(\bar{x}(\cdot),\bar{u}(\cdot))$ is a solution to  (\ref{eq:-8}).
Let $(p(\cdot),q(\cdot))$ and $(P(\cdot),Q(\cdot))$ be the solutions
of the \emph{adjoint equations}
\begin{equation}
\left\{ \begin{aligned}dp(t) & =-[\partial_{x}b_{1}(t)p(t)-\partial_{x}f_{1}(t)]dt\\
 & \quad-[\partial_{x}b_{2}(t)p(t)+\partial_{x}\sigma(t)q(t)-\partial_{x}f_{2}(t)]d\langle W\rangle_{t}\\
 & \quad+q(t)dW_{t},\qquad t\in[0,T],\;\mathbb{P}_{\lambda}\text{-a.s.},\\
p(T) & =-\partial_{x}h(\bar{x}(T)),
\end{aligned}
\right.\label{eq:-16}
\end{equation}
and
\begin{equation}
\left\{ \begin{aligned}dP(t) & =-[2\partial_{x}b_{1}(t)P(t)+\partial_{x}^{2}b_{1}(t)p(t)-\partial_{x}^{2}f_{1}(t)]dt\\
 & \quad-\big[\big(2\partial_{x}b_{2}(t)+\partial_{x}\sigma(t)^{2}\big)P(t)+\partial_{x}\sigma(t)Q(t)+\partial_{x}^{2}b_{2}(t)p(t)+\partial_{x}^{2}\sigma(t)q(t)-\partial_{x}^{2}f_{2}(t)\big]d\langle W\rangle_{t}\\
 & \quad+Q(t)dW_{t},\qquad t\in[0,T],\;\mathbb{P}_{\lambda}\text{-a.s.},\\
P(T) & =-\partial_{x}^{2}h(\bar{x}(T)),
\end{aligned}
\right.\label{eq:-17}
\end{equation}
and let $H_{1}(t,x,u),\;H_{2}(t,x,u)$ be the \emph{Hamiltonians}
defined by
\[
H_{1}(t,x,u)\triangleq b_{1}(t,x,u)p(t)-f_{1}(t,x,u),
\]
\[
H_{2}(t,x,u)\triangleq b_{2}(t,x,u)p(t)+\sigma(t,x,u)q(t)-f_{2}(t,x,u)+\frac{1}{2}[\sigma(t,x,u)-\sigma(t,x,\bar{u}(t))]^{2}P(t).
\]
Then
\begin{equation}
\left\{ \begin{aligned}H_{1}(t,\bar{x}(t),\bar{u}(t)) & =\max_{u\in\mathbb{U}}H_{1}(t,\bar{x}(t),u), & \mathfrak{M}_{1}\text{-a.e.},\\
H_{2}(t,\bar{x}(t),\bar{u}(t)) & =\max_{u\in\mathbb{U}}H_{2}(t,\bar{x}(t),u), & \mathfrak{M}_{2}\text{-a.e.},
\end{aligned}
\right.\label{eq:-38}
\end{equation}
\end{thm}
\begin{rem}
\label{rem:-2}(i) Notice that the assumptions $\hyperlink{A.1}{\text{(A.1)}}$
and $\hyperlink{A.2}{\text{(A.2)}}$ imply that $\varphi,\partial_{x}\varphi,\partial_{x}^{2}\varphi$
are uniformly bounded for $\varphi=b_{1},b_{2},\sigma,f_{1},f_{2},h$.
Indeed, the assumption $\hyperlink{A.1}{\text{(A.1)}}$ implies the
uniform boundedness of $\varphi$. The boundedness of $\partial_{x}\varphi(t,0,u)$
for $(t,u)$ can also be deduced from $\hyperlink{A.1}{\text{(A.1)}}$
with $\hat{u}=u$, which together with $\hyperlink{A.2}{\text{(A.2)}}$
implies the uniform boundedness of $\partial_{x}\varphi$. Similarly,
$\partial_{x}^{2}\varphi$ is also uniformly bounded. 

(ii) The adjoint equations (\ref{eq:-16}) and (\ref{eq:-17}) are
introduced in order to reduce the general case with a non-trivial
$h(x(T))$ in the cost functional $J(u)$ to the one without an $h(x(T))$
term. In other words, it transforms the cost $h(x(T))$ at terminal
time into a cumulative cost over the interval $[0,T]$. This can be
seen more clearly from the proof of Theorem \ref{thm:}.
\end{rem}

\section{\label{sec:-3}Proof of the stochastic maximum principle}

In this section, we prove Theorem \ref{thm:} for the optimization
problem (\ref{eq:-8}) on the Sierpinski gasket. Our argument is based
on the idea of approximation and duality used in the paper \citep{Pen90}
and the monograph \citep{Yong99} for classical Euclidean setting,
while overcoming some difficulties concerning the driver martingale
$W$ on the Sierpinski gasket. More specifically, as we shall see,
a crucial ingredient of our argument is an order comparison lemma
(Lemma \ref{lem:-1}), which is needed for stochastic Taylor expansions.
Another technical lemma crucial to the proof of Theorem \ref{thm:}
is Lemma \ref{lem:-2}, which gives the orders of approximation errors.
\begin{defn}
\label{def:}Let $\lambda\in\mathcal{P}(\mathbb{S}),\,k\ge1$ and
$E\in\mathcal{B}([0,\infty)\times\Omega)$ be a progressively measurable
set. For each $I\in\mathcal{B}([0,\infty))$, we denote
\[
m_{k,\lambda}(I;E)=\mathbb{E}_{\lambda}\Big[\Big(\int_{I}1_{E}(t,\omega)d\langle W\rangle_{t}\Big)^{k}\Big].
\]
Clearly, the map $I\mapsto|I|+m_{1,\lambda}(I;\Omega)$ is a Borel
measure on $\mathcal{B}([0,\infty))$, where $|\cdot|$ is the one-dimensional
Lebesgue measure. We denote by $\mathcal{B}_{\lambda}([0,\infty))$
the completion of $\mathcal{B}([0,\infty))$ with respect to the measure
$|\cdot|+m_{1,\lambda}(\cdot\,;\Omega)$.
\end{defn}
\begin{lem}
\label{lem:-1}Let $\lambda\in\mathcal{P}(\mathbb{S})$, and $E\in\mathcal{B}([0,\infty)\times\Omega)$
be a progressively measurable set. Let $\{I_{\epsilon}\}_{\epsilon>0}$
be a family of $\mathcal{B}_{\lambda}([0,\infty))$-measurable subsets
of $[0,\infty)$ such that $\lim_{\epsilon\to0}|I_{\epsilon}|=0$.
Then, for some universal constant $C_{\ast}>0$, 
\begin{equation}
m_{k+1,\lambda}(I_{\epsilon};E)\le C_{\ast}(k+1)\,|I_{\epsilon}|^{1-d_{s}/2}m_{k,\lambda}(I_{\epsilon};E),\;\;\text{for all}\;k\in\mathbb{N}_{+}.\label{eq:-23}
\end{equation}
In particular, 
\begin{equation}
m_{l,\lambda}(I_{\epsilon};E)=\mathrm{o}(m_{k,\lambda}(I_{\epsilon};E)),\;\;\text{as}\;\epsilon\to0,\label{eq:-4}
\end{equation}
for all $k\in\mathbb{N}_{+}$ and $l>k$. 
\end{lem}
\begin{proof}
Let $\phi_{\epsilon}(t)=\phi_{\epsilon}(t,\omega)=1_{I_{\epsilon}}(t)1_{E}(t,\omega)$.
Then, for each $\epsilon>0$, $\phi_{\epsilon}$ is a bounded progressively
measurable process. Clearly, we have the following iterated integral
representation 
\begin{equation}
\mathbb{E}_{\lambda}\Big[\Big(\int_{0}^{\infty}\phi_{\epsilon}(t)d\langle W\rangle_{t}\Big)^{k}\Big]=\mathbb{E}_{\lambda}\Big[k!\int_{0<t_{1}<\cdots<t_{k}<\infty}\phi_{\epsilon}(t_{1})\cdots\phi_{\epsilon}(t_{k})d\langle W\rangle_{t_{1}}\cdots d\langle W\rangle_{t_{k}}\Big].\label{eq:-39}
\end{equation}
Since $\phi_{\epsilon}$ is progressively measurable, we have $\phi_{\epsilon}(t)\in\mathcal{F}_{t}^{\lambda}$.
Therefore, by (\ref{eq:-39}) and the tower property, 
\begin{equation}
\begin{aligned} & \;\quad\mathbb{E}_{\lambda}\Big[\Big(\int_{0}^{\infty}\phi_{\epsilon}(t)d\langle W\rangle_{t}\Big)^{k+1}\Big]\\
 & =\mathbb{E}_{\lambda}\Big[(k+1)!\int_{0<t_{1}<\cdots<t_{k+1}<\infty}\phi_{\epsilon}(t_{1})\cdots\phi_{\epsilon}(t_{k+1})d\langle W\rangle_{t_{1}}\cdots d\langle W\rangle_{t_{k+1}}\Big]\\
 & =\mathbb{E}_{\lambda}\Big[(k+1)!\int_{0<t_{1}<\cdots<t_{k}<\infty}\phi_{\epsilon}(t_{1})\cdots\phi_{\epsilon}(t_{k})\mathbb{E}_{\lambda}\Big(\int_{t_{k}}^{\infty}\phi_{\epsilon}(t_{k+1})d\langle W\rangle_{t_{k+1}}\Big|\,\mathcal{F}_{t_{k}}^{\lambda}\Big)d\langle W\rangle_{t_{1}}\cdots d\langle W\rangle_{t_{k}}\Big].
\end{aligned}
\label{eq:-3}
\end{equation}
Recall that $\phi_{\epsilon}(t,\omega)\le1_{I_{\epsilon}}(t)$. By
\citep[Lemma 4.17]{LQ16}, we have 
\begin{equation}
\begin{aligned}\mathbb{E}_{\lambda}\Big(\int_{t_{k}}^{\infty}\phi_{\epsilon}(t_{k+1})d\langle W\rangle_{t_{k+1}}\Big|\,\mathcal{F}_{t_{k}}^{\lambda}\Big) & \le\mathbb{E}_{\lambda}\Big(\int_{t_{k}}^{\infty}1_{I_{\epsilon}}(t_{k+1})d\langle W\rangle_{t_{k+1}}\Big|\,\mathcal{F}_{t_{k}}^{\lambda}\Big)\\
 & =\int_{t_{k}}^{\infty}1_{I_{\epsilon}}(t_{k+1})(P_{t_{k+1}-t_{k}}\mu)(X_{t_{k}})dt_{k+1},
\end{aligned}
\label{eq:-2}
\end{equation}
where, for any Borel measure $\lambda$ on $\mathbb{S}$,
\[
P_{t}\lambda(x)\triangleq\int_{\mathbb{S}}p_{t}(x,y)\,\lambda(dy),\;\;x\in\mathbb{S},
\]
with $p_{t}(x,y)$ being the transition kernel of $\{X_{t}\}_{t\ge0}$,
which is jointly continuous on $\mathbb{S}\times\mathbb{S}$. By \citep[Theorem 5.3.1]{Ki01},
there exists a universal constant $C_{\ast}>0$ such that
\begin{equation}
C_{\ast}^{-1}\max\{1,t^{-d_{s}/2}\}\le p_{t}(x,y)\le C_{\ast}\max\{1,t^{-d_{s}/2}\},\;\;t\in(0,\infty),\,x,y\in\mathbb{S},\label{eq:-53}
\end{equation}
where $d_{s}=2\log3/\log5<2$ is the spectral dimension of $\{X_{t}\}_{t\ge0}$.
Therefore,
\[
\Vert P_{t}\mu\Vert_{L^{\infty}}\le C_{\ast}\max\{1,t^{-d_{s}/2}\},\quad t\in(0,\infty).
\]
For $I_{\epsilon}$ with $|I_{\epsilon}|\le1$, by (\ref{eq:-2}),
\[
\begin{aligned} & \;\quad\mathbb{E}_{\lambda}\Big(\int_{t_{k}}^{\infty}\phi_{\epsilon}(t_{k+1})d\langle W\rangle_{t_{k+1}}\Big|\,\mathcal{F}_{t_{k}}^{\lambda}\Big)\\
 & \le\int_{t_{k}}^{\infty}1_{I_{\epsilon}}(t_{k+1})(t_{k+1}-t_{k})^{-d_{s}/2}\,dt_{k+1}\\
 & \le\int_{t_{k}}^{t_{k}+|I_{\epsilon}|}(t_{k+1}-t_{k})^{-d_{s}/2}\,dt_{k+1}+|I_{\epsilon}|^{-d_{s}/2}\int_{0}^{\infty}1_{I_{\epsilon}}(t_{k+1})\,dt_{k+1}\\
 & =C_{\ast}|I_{\epsilon}|^{1-d_{s}/2}.
\end{aligned}
\]
Hence, by (\ref{eq:-3}), 
\[
\mathbb{E}_{\lambda}\Big[\Big(\int_{0}^{\infty}\phi_{\epsilon}(t)d\langle W\rangle_{t}\Big)^{k+1}\Big]\le C_{\ast}|I_{\epsilon}|^{1-d_{s}/2}\mathbb{E}_{\lambda}\Big[(k+1)!\int_{0<t_{1}<\cdots<t_{k}<\infty}\phi_{\epsilon}(t_{1})\cdots\phi_{\epsilon}(t_{k})d\langle W\rangle_{t_{1}}\cdots d\langle W\rangle_{t_{k}}\Big].
\]
By (\ref{eq:-39}) again, we conclude that 
\[
\mathbb{E}_{\lambda}\Big[\Big(\int_{0}^{\infty}\phi_{\epsilon}(t)d\langle W\rangle_{t}\Big)^{k+1}\Big]\le C_{\ast}(k+1)|I_{\epsilon}|^{1-d_{s}/2}\mathbb{E}_{\lambda}\Big[\Big(\int_{0}^{\infty}\phi_{\epsilon}(t)d\langle W\rangle_{t}\Big)^{k}\Big],
\]
which is (\ref{eq:-23}).

When $l$ is an integer, the asymptotic (\ref{eq:-4}) is a direct
corollary of (\ref{eq:-23}). For real-valued $l>k$, the conclusion
follows easily from interpolation 
\[
m_{k+\theta,\lambda}(I_{\epsilon};E)\le m_{k,\lambda}(I_{\epsilon};E)^{1-\theta}m_{k+1,\lambda}(I_{\epsilon};E)^{\theta},\;\;\theta\in(0,1).
\]
\end{proof}
\begin{rem}
\label{rem:-1}The order estimate (\ref{eq:-23}) implies that $m_{k,\lambda}(I_{\epsilon};E)=\mathrm{O}(|I_{\epsilon}|^{k(1-d_{s}/2)})$,
which is quite sharp. In fact, since the heat kernel estimate (\ref{eq:-53})
is two-sided, by \citep[Lemma 4.17]{LQ16}, we have that 
\[
\mathbb{E}_{x}\Big(\int_{0}^{\epsilon}d\langle W\rangle_{t}\Big)=\int_{0}^{\epsilon}P_{t}\mu(x)dt\ge C_{\ast}\epsilon^{1-d_{s}/2},\;\;\text{for all}\;x\in\mathbb{S}.
\]
Therefore, by (\ref{eq:-23}), 
\[
\mathbb{E}_{x}\Big[\Big(\int_{0}^{\epsilon}d\langle W\rangle_{t}\Big)^{k}\Big]^{1/k}\le\mathrm{O}(\epsilon^{1-d_{s}/2})\le\mathbb{E}_{x}\Big(\int_{0}^{\epsilon}d\langle W\rangle_{t}\Big).
\]
Notice that the reverse of the above inequality is a direct consequence
of Hölder's inequality. Therefore, we see that, up to a multiplicative
constant, 
\[
\mathbb{E}_{\lambda}\Big[\Big(\int_{0}^{\epsilon}d\langle W\rangle_{t}\Big)^{k}\Big]\sim\mathbb{E}_{\lambda}\Big[\Big(\int_{0}^{\epsilon}d\langle W\rangle_{t}\Big)\Big]^{k}\sim\epsilon^{k(1-d_{s}/2)},\;\;\text{as}\;\epsilon\to0.
\]
\end{rem}
We shall also need the following estimate for solutions of linear
SDEs driven by the Brownian martingale $W$.
\begin{lem}
\label{lem:}Let $\alpha_{1}\in L^{\infty}(\mathfrak{M}_{1}),\alpha_{2}\in L^{\infty}(\mathfrak{M}_{2})$
be progressively measurable processes, and $\beta\in L^{\infty}(\mathfrak{M}_{2})$
be a predictable process. Let $\{Y_{t}\}$ be the solution to the
SDE
\[
\left\{ \begin{aligned}dY_{t} & =(a_{1}(t)Y_{t}+\alpha_{1}(t))dt+(a_{2}(t)Y_{t}+\alpha_{2}(t))d\langle W\rangle_{t}\\
 & \quad+(b(t)Y_{t}+\beta(t))dW_{t},\qquad t\in[0,T],\\
Y_{0} & =\xi.
\end{aligned}
\right.
\]
Suppose that
\[
|\varphi(t)|\le M\quad\text{for }\varphi=a_{1},a_{2},b,
\]
where $M>0$ is a constant. Then, for each $\lambda\in\mathcal{P}(\mathbb{S})$
and each $k\in(1/2,\infty)$, 
\begin{equation}
\mathfrak{T}_{2k}(Y)\le C\,\mathbb{E}_{\lambda}\bigg[|\xi|^{2k}+\Big(\int_{0}^{T}|\alpha_{1}(t)|dt\Big)^{2k}+\Big(\int_{0}^{T}|\alpha_{2}(t)|d\langle W\rangle_{t}\Big)^{2k}+\Big(\int_{0}^{T}|\beta(t)|^{2}d\langle W\rangle_{t}\Big)^{k}\bigg].\label{eq:-44}
\end{equation}
where $C=C(k,M)>0$ is a constant depending only on $k,M$, 
\begin{equation}
\mathfrak{T}_{2k}(\varphi)=\mathbb{E}_{\lambda}\Big(\sup_{0\le t\le T}|\varphi(t)|^{2k}\mathrm{e}_{t}^{-1}+\int_{0}^{T}|\varphi(t)|^{2k}\mathrm{e}_{t}^{-1}d\langle W\rangle_{t}\Big).\label{eq:-46}
\end{equation}
for any $k\ge1$ and any progressively measurable process $\varphi(t)$,
and 
\begin{equation}
\mathrm{e}_{t}=\exp(\kappa\langle W\rangle_{t})\label{eq:-47}
\end{equation}
for a sufficiently large constant $\kappa>0$ depending only on $k,M$
(e.g. $\kappa=8k^{2}(M+1)^{2}$ will suffice). Therefore, 
\begin{equation}
\begin{aligned}\mathfrak{T}_{2k}(Y) & \le C\,\bigg\{\mathbb{E}_{\lambda}(|\xi|^{2k})+\Vert\alpha_{1}\Vert_{L^{\infty}(\mathfrak{M}_{1})}^{2k}\mathbb{E}_{\lambda}\Big[\Big(\int_{0}^{T}1\{\alpha_{1}(t)\not=0\}dt\Big)^{2k}\Big]\\
 & \quad+\Vert\alpha_{2}\Vert_{L^{\infty}(\mathfrak{M}_{2})}^{2k}\mathbb{E}_{\lambda}\Big[\Big(\int_{0}^{T}1\{\alpha_{2}(t)\not=0\}d\langle W\rangle_{t}\Big)^{2k}\Big]\\
 & \quad+\Vert\beta\Vert_{L^{\infty}(\mathfrak{M}_{2})}^{2k}\mathbb{E}_{\lambda}\Big[\Big(\int_{0}^{T}1\{\beta(t)\not=0\}d\langle W\rangle_{t}\Big)^{k}\Big]\bigg\},
\end{aligned}
\label{eq:-1}
\end{equation}
\end{lem}
\begin{rem}
\label{rem:-3}From now on, for the ease of notation, we shall use
the same notation $\mathrm{e}_{t}$ to denote $\exp(\kappa\langle W\rangle_{t})$
with possibly different constants $\kappa$ depending only on $k$
and the $L^{\infty}$ norms of coefficients of SDEs. 
\end{rem}
\begin{proof}
To simplify notation, we shall denote $Y^{m}=|Y|^{m}\,\mathrm{sgn}(Y)$
for any $m>0$. By Itô's formula, 
\[
\begin{aligned} & \;\quad|Y_{t}|^{2k}\mathrm{e}_{t}^{-1}\\
 & =|\xi|^{2k}+\int_{0}^{t}\big[2kY_{r}^{2k-1}(a_{1}(r)Y_{r}+\alpha_{1}(r))-k_{1}|Y_{r}|^{2k}\big]\mathrm{e}_{r}^{-1}dr\\
 & \quad+\int_{0}^{t}\big[2kY_{r}^{2k-1}(a_{2}(r)Y_{r}+\alpha_{2}(r))+k(2k-1)Y_{r}^{2k-2}(b(r)Y_{r}+\beta(r))^{2}-k_{2}|Y_{r}|^{2k}\big]\mathrm{e}_{r}^{-1}d\langle W\rangle_{r}\\
 & \quad+\int_{0}^{t}2kY_{r}^{2k-1}(b(r)Y_{r}+\beta(r))\mathrm{e}_{r}^{-1}dW_{r}\\
 & \le|\xi|^{2k}+\int_{0}^{t}2k|Y_{r}|^{2k-1}|\alpha_{1}(r)|\mathrm{e}_{r}^{-1}dr\\
 & \quad+\int_{0}^{t}\big[(2kM+4k^{2}M^{2}-k_{2})|Y_{r}|^{2k}+2k|Y_{r}|^{2k-1}|\alpha_{2}(r)|+4k^{2}|Y_{r}|^{2k-2}|\beta(r)|^{2}\big]\mathrm{e}_{r}^{-1}d\langle W\rangle_{r}\\
 & \quad+\bigg|\int_{0}^{t}2kY_{r}^{2k-1}(b(r)Y_{r}+\beta(r))\mathrm{e}_{r}^{-1}dW_{r}\bigg|.
\end{aligned}
\]
Denote $Z=\sup_{0\le t\le T}|Y_{t}|\mathrm{e}_{t}^{-1/(2k)}$. Then
\begin{equation}
\begin{aligned}Z^{2k} & \le|\xi|^{2k}+2kZ^{2k-1}\bigg(\int_{0}^{T}|\alpha_{1}(t)|dt+\int_{0}^{T}|\alpha_{2}(t)|d\langle W\rangle_{t}\bigg)\\
 & \quad+4k^{2}Z^{2k-2}\int_{0}^{T}|\beta(t)|^{2}d\langle W\rangle_{t}+(2kM+4k^{2}M^{2}-k_{2})\int_{0}^{T}|Y_{r}|^{2k}\mathrm{e}_{r}^{-1}d\langle W\rangle_{r}\\
 & \quad+\sup_{0\le t\le T}\bigg|\int_{0}^{t}2kY_{r}^{2k-1}(b(r)Y_{r}+\beta(r))\mathrm{e}_{r}^{-1}dW_{r}\bigg|.
\end{aligned}
\label{eq:}
\end{equation}
By the Burkholder\textendash Davis\textendash Gundy inequality,
\[
\begin{aligned} & \;\quad\mathbb{E}_{\lambda}\bigg(\sup_{0\le t\le T}\bigg|\int_{0}^{t}Y_{r}^{2k-1}(b(r)Y_{r}+\beta(r))\mathrm{e}_{r}^{-1}dW_{r}\bigg|\bigg)\\
 & \le C_{\ast}\,\mathbb{E}_{\lambda}\bigg[\bigg(\int_{0}^{T}\big(M^{2}|Y_{r}|^{4k}+|Y_{r}|^{4k-2}|\beta(r)|^{2}\big)\mathrm{e}_{r}^{-2}d\langle W\rangle_{r}\bigg)^{1/2}\bigg]\\
 & \le C_{\ast}\,\mathbb{E}_{\lambda}\bigg[MZ^{k}\bigg(\int_{0}^{T}|Y_{r}|^{2k}\mathrm{e}_{r}^{-1}d\langle W\rangle_{r}\bigg)^{1/2}+Z^{2k-1}\bigg(\int_{0}^{T}|\beta(r)|^{2}d\langle W\rangle_{r}\bigg)^{1/2}\bigg]\\
 & \le C_{\ast}\,\mathbb{E}_{\lambda}\bigg[(\epsilon_{1}+\epsilon_{2})Z^{2k}+\frac{M}{\epsilon_{1}}\int_{0}^{T}|Y_{r}|^{2k}\mathrm{e}_{r}^{-1}d\langle W\rangle_{r}+\frac{1}{2k\epsilon_{2}^{2k-1}}\bigg(\int_{0}^{T}|\beta(r)|^{2}d\langle W\rangle_{r}\bigg)^{k}\bigg].
\end{aligned}
\]
where $C_{\ast}>0$ is a universal constant. Choosing $\epsilon_{1}=1/4$
and $\epsilon_{2}>0$ sufficiently small gives
\[
\begin{aligned} & \;\quad\mathbb{E}_{\lambda}\bigg(\sup_{0\le t\le T}\bigg|\int_{0}^{t}Y_{r}^{2k-1}(b(r)Y_{r}+\beta(r))\mathrm{e}_{r}^{-1}dW_{r}\bigg|\bigg)\\
 & \le\frac{1}{2}\mathbb{E}_{\lambda}\big(Z^{2k}\big)+\mathbb{E}_{\lambda}\bigg[4M\int_{0}^{T}|Y_{r}|^{2k}\mathrm{e}_{r}^{-1}d\langle W\rangle_{r}+C\,\bigg(\int_{0}^{T}|\beta(r)|^{2}d\langle W\rangle_{r}\bigg)^{k}\bigg],
\end{aligned}
\]
where $C>0$ denotes a constant depending only on $k,M$. Since $\kappa>4M+2kM+4k^{2}M^{2}$,
(\ref{eq:-1}) follows easily from the above and (\ref{eq:}) and
Young's inequality. Notice that, in the above inequality, we have
used the fact that $\mathbb{E}(Z^{2k})<\infty$ (or alternatively
an localization argument together with $|Z|<\infty$ a.s.), which
can be shown by an iteration argument similar to the proof of \citep[Theorem 3.10]{LQ16}. 
\end{proof}
We now turn to the derivation of the stochastic maximum principle.
Suppose that $\bar{u}\in\mathcal{A}[0,T]$ is a minimizer of (\ref{eq:-8}),
and $\bar{x}(\cdot)$ is the corresponding controlled process. Let
$\{I_{\epsilon}\}_{\epsilon>0}$ be an arbitrary family of $\mathcal{B}_{\lambda}([0,\infty))$-measurable
subsets of $[0,T]$ such that $\lim_{\epsilon\to0}|I_{\epsilon}|=0$.

Let $S_{1},S_{2}\subseteq[0,\infty)\times\Omega$ be disjoint optional
sets such that $\mathfrak{M}_{1}$ is supported on $S_{1}$ and $\mathfrak{M}_{2}$
on $S_{2}$. An example of such $(S_{1},S_{2})$ is $S_{1}=1\{(t,\omega):L_{1}(t,\omega)=1\}$,
$S_{2}=1\{(t,\omega):L_{2}(t,\omega)=1\}$, where 
\[
L_{1}=\frac{d\mathfrak{M}_{1}}{d(\mathfrak{M}_{1}+\mathfrak{M}_{2})},\;L_{2}(t)=\frac{d\mathfrak{M}_{2}}{d(\mathfrak{M}_{1}+\mathfrak{M}_{2})}
\]
are the Radon\textendash Nikodym derivatives with respect to the optional
$\sigma$-field on $[0,T]\times\Omega$. For arbitrary $u_{1},u_{2}\in\mathcal{A}[0,T]$,
let
\[
u^{\epsilon}(t,\omega)=\left\{ \begin{aligned}\bar{u}(t,\omega), & \;\;\text{if }(t,\omega)\in([0,T]\backslash I_{\epsilon})\times\Omega,\\
u_{1}(t,\omega), & \;\;\text{if }(t,\omega)\in(I_{\epsilon}\times\Omega)\cap S_{1},\\
u_{2}(t,\omega), & \;\;\text{if }(t,\omega)\in(I_{\epsilon}\times\Omega)\cap S_{2}.
\end{aligned}
\right.
\]
Let 
\begin{equation}
E=\{(t,\omega)\in S_{1}:\bar{u}(t,\omega)\not=u_{1}(t,\omega)\}\cup\{(t,\omega)\in S_{2}:\bar{u}(t,\omega)\not=u_{2}(t,\omega)\}.\label{eq:-40}
\end{equation}
Then $E$ is progressively measurable. Notice that if $\mathfrak{M}_{2}(E)=0$,
then $m_{k,\lambda}([0,\infty);E)=0$ for all $k\in\mathbb{N}_{+}$.

We denote by $x^{\epsilon}(\cdot)$ the controlled process corresponding
to $u^{\epsilon}(\cdot)$, and let
\begin{equation}
\xi^{\epsilon}=x^{\epsilon}-\bar{x}.\label{eq:-24}
\end{equation}
Define the first-order approximating process $y^{\epsilon}(\cdot)$
by
\begin{equation}
\left\{ \begin{aligned}dy^{\epsilon}(t) & =\partial_{x}b_{1}(t)y^{\epsilon}(t)dt+\partial_{x}b_{2}(t)y^{\epsilon}(t)d\langle W\rangle_{t}+\big(\delta\sigma(t)+\partial_{x}\sigma(t)y^{\epsilon}(t)\big)dW_{t},\\
y^{\epsilon}(0) & =0,
\end{aligned}
\right.\label{eq:-9}
\end{equation}
and the second-order approximating process $z^{\epsilon}(\cdot)$
by
\begin{equation}
\left\{ \begin{aligned}dz^{\epsilon}(t) & =\Big[\partial_{x}b_{1}(t)z^{\epsilon}(t)+\delta b_{1}(t)+\frac{1}{2}\partial_{x}^{2}b_{1}(t)y^{\epsilon}(t)^{2}\Big]dt\\
 & \quad+\Big[\partial_{x}b_{2}(t)z^{\epsilon}(t)+\delta b_{2}(t)+\frac{1}{2}\partial_{x}^{2}b_{2}(t)y^{\epsilon}(t)^{2}\Big]d\langle W\rangle_{t}\\
 & \quad+\Big[\partial_{x}\sigma(t)z^{\epsilon}(t)+\delta(\partial_{x}\sigma)(t)y^{\epsilon}(t)+\frac{1}{2}\partial_{x}^{2}\sigma(t)y^{\epsilon}(t)^{2}\Big]dW_{t},\\
z^{\epsilon}(0) & =0,
\end{aligned}
\right.\label{eq:-10}
\end{equation}
where, for any function $\varphi:[0,\infty)\times\mathbb{R}\times\mathbb{U}\to\mathbb{R}$,
we denote
\[
\varphi(t)=\varphi(t,\bar{x}(t),\bar{u}(t)),\;\delta\varphi(t)=\varphi(t,\bar{x}(t),u^{\epsilon}(t))-\varphi(t,\bar{x}(t),\bar{u}(t)),\;\;t\ge0.
\]
Clearly, $\mathrm{supp}(\delta\varphi)\subseteq I_{\epsilon}$. We
shall need the following estimates.
\begin{lem}
\label{lem:-2}Let $E$ be the progressively measurable set defined
by (\ref{eq:-40}). Then, for each $k\ge1$, as $\epsilon\to0$,
\begin{align}
\mathfrak{T}_{2k}(\xi^{\epsilon}) & =\mathfrak{M}_{1}(E)\mathrm{O}\big(|I_{\epsilon}|^{k})+\mathrm{O}\big(m_{k,\lambda}(I_{\epsilon};E)\big),\label{eq:-13}\\
\mathfrak{T}_{2k}(y^{\epsilon}) & =\mathfrak{M}_{1}(E)\mathrm{O}\big(|I_{\epsilon}|^{k})+\mathrm{O}\big(m_{k,\lambda}(I_{\epsilon};E)\big),\label{eq:-14}\\
\mathfrak{T}_{2k}(z^{\epsilon}) & =\mathfrak{M}_{1}(E)\mathrm{O}\big(|I_{\epsilon}|^{2k})+\mathrm{O}\big(m_{2k,\lambda}(I_{\epsilon};E)\big),\label{eq:-15}\\
\mathfrak{T}_{2k}(\xi^{\epsilon}(t) & -y^{\epsilon}(t))=\mathfrak{M}_{1}(E)\mathrm{O}\big(|I_{\epsilon}|^{2k})+\mathrm{O}\big(m_{2k,\lambda}(I_{\epsilon};E)\big),\label{eq:-11}\\
\mathfrak{T}_{2k}(\xi^{\epsilon}(t) & -y^{\epsilon}(t)-z^{\epsilon}(t))=\mathfrak{M}_{1}(E)\,\mathrm{o}\big(|I_{\epsilon}|^{2k})+\mathrm{o}\big(m_{2k,\lambda}(I_{\epsilon};E)\big).\label{eq:-12}
\end{align}
\end{lem}
\begin{proof}
We only present the proof of (\ref{eq:-13}) and (\ref{eq:-11}),
since the proof of (\ref{eq:-14}) is similar to that of (\ref{eq:-13}),
while the proof of (\ref{eq:-15}) and (\ref{eq:-12}) are similar
to that of (\ref{eq:-11}). The difference between the proof of (\ref{eq:-13})
and (\ref{eq:-11}) is that the SDE for $\xi^{\epsilon}-y^{\epsilon}$
involves $\xi^{\epsilon}$ as bias terms $\alpha_{1},\alpha_{2},\beta$
in Lemma \ref{lem:} (see (\ref{eq:-50}) below), which requires further
estimate. This is also the case for $z^{\epsilon}$ and $\xi^{\epsilon}-y^{\epsilon}-z^{\epsilon}$,
and hence their estimates are similar to that of $\xi^{\epsilon}-y^{\epsilon}$. 

For any function $\varphi:[0,\infty)\times\mathbb{R}\times\mathbb{U}\to\mathbb{R}$,
denote
\[
\tilde{\varphi}(t)=\int_{0}^{1}\varphi(t,(1-\theta)\bar{x}(t)+\theta x^{\epsilon}(t),u^{\epsilon}(t))\,d\theta,\quad t\ge0.
\]
By (\ref{eq:-7}),
\[
\left\{ \begin{aligned}d\xi^{\epsilon}(t) & =\big[\widetilde{\partial_{x}b_{1}}(t)\xi^{\epsilon}(t)+\delta b_{1}(t)\big]dt+\big[\widetilde{\partial_{x}b_{2}}(t)\xi^{\epsilon}(t)+\delta b_{2}(t)\big]d\langle W\rangle_{t}\\
 & \quad+\big[\widetilde{\partial_{x}\sigma}(t)\xi^{\epsilon}(t)+\delta\sigma(t)\big]dW_{t},\qquad t\ge0,\\
\xi^{\epsilon}(0) & =0.
\end{aligned}
\right.
\]
Let $E_{\epsilon}=E\cap(I_{\epsilon}\times\Omega)$. Then $\mathrm{supp}(\delta\varphi)\subseteq E_{\epsilon},\;\varphi=b_{1},b_{2},\sigma$.
By Lemma \ref{lem:},
\[
\begin{aligned} & \;\quad\mathbb{E}_{\lambda}\Big(\sup_{0\le t\le T}|\xi^{\epsilon}(t)|^{2k}\mathrm{e}_{t}^{-1}\Big)\\
 & \le C\,\mathbb{E}_{\lambda}\bigg[\Big(\int_{0}^{T}|\delta b_{1}(t)|dt\Big)^{2k}+\Big(\int_{0}^{T}|\delta b_{2}(t)|d\langle W\rangle_{t}\Big)^{2k}+\Big(\int_{0}^{T}|\delta\sigma(t)|^{2}d\langle W\rangle_{t}\Big)^{k}\bigg]\\
 & =\mathfrak{M}_{1}(E)\mathrm{O}\big(|I_{\epsilon}|^{k})+\mathrm{O}\big(m_{k,\lambda}(I_{\epsilon};E)\big),
\end{aligned}
\]
where $C>0$ denotes a constant depending only on $k,M$, but might
be different at various appearances. We should point out that we explicitly
include the term $\mathfrak{M}_{1}(E)$ in the last equation to reflect
the fact that $\mathbb{E}_{\lambda}\big[\big(\int_{0}^{T}|\delta b_{1}(t)|dt\big)^{2k}\big]=0$
whenever $\mathfrak{M}_{1}(E)=0$. The appearances of $\mathfrak{M}_{1}(E)$
in other estimates are out of the same purpose. The above estimate
completes the proof of (\ref{eq:-13}). The proof of (\ref{eq:-14})
is similar. 

We now turn to the proof of (\ref{eq:-11}). By the definition of
$\tilde{\varphi}(t)$, we have 
\begin{equation}
\tilde{\varphi}(t)-\varphi(t)=\delta\varphi(t)+\mathrm{O}(|\xi^{\epsilon}|)=1_{E_{\epsilon}}(t)\mathrm{O}(1)+\mathrm{O}(|\xi^{\epsilon}|).\label{eq:-43}
\end{equation}
Let $\eta^{\epsilon}=\xi^{\epsilon}-y^{\epsilon}$, and 
\[
\chi_{1}(t)=1_{E_{\epsilon}}(t)\mathrm{O}(|\xi^{\epsilon}(t)|)+\mathrm{O}(|\xi^{\epsilon}(t)|^{2}).
\]
Then, by (\ref{eq:-43}) and the fact that $\delta\varphi=1_{E_{\epsilon}}$
for $\varphi=b_{1},b_{2},\sigma$, we have 
\begin{equation}
\begin{aligned}d\eta^{\epsilon} & =[\partial_{x}b_{1}(t)\eta^{\epsilon}+1_{E_{\epsilon}}(t)\mathrm{O}(1)+\chi_{1}(t)]\,dt\\
 & \quad+[\partial_{x}b_{2}(t)\eta^{\epsilon}+1_{E_{\epsilon}}(t)\mathrm{O}(1)+\chi_{1}(t)]\,d\langle W\rangle_{t}\\
 & \quad+[\partial_{x}\sigma(t)\eta^{\epsilon}+\chi_{1}(t)]\,dW_{t}.
\end{aligned}
\label{eq:-50}
\end{equation}
In order to apply Lemma \ref{lem:}, since the desired estimates involving
$1_{E_{\epsilon}}(t)\mathrm{O}(1)$ follow directly from definition,
we need to estimate $\mathbb{E}_{\lambda}\big[\big(\int_{0}^{T}\chi_{1}(t)\,dt\big)^{2k}\big]$,
$\mathbb{E}_{\lambda}\big[\big(\int_{0}^{T}\chi_{1}(t)\,d\langle W\rangle_{t}\big)^{2k}\big]$
and $\mathbb{E}_{\lambda}\big[\big(\int_{0}^{T}\chi_{1}(t)^{2}\,d\langle W\rangle_{t}\big)^{k}\big]$,
where 
\[
\chi(t)=1_{E_{\epsilon}}(t)\mathrm{O}(|\xi^{\epsilon}(t)|)+\mathrm{O}(|\xi^{\epsilon}(t)|^{2}).
\]

We first estimate $\mathbb{E}_{\lambda}\big[\big(\int_{0}^{T}\chi_{1}(t)\,dt\big)^{2k}\big]$.
Notice that, by Lemma \ref{lem:-3} and (\ref{eq:-13}), 
\begin{equation}
\begin{aligned}\mathbb{E}_{\lambda}\Big[\Big(\int_{0}^{T}1_{E}(t)|\xi^{\epsilon}(t)|\,dt\Big)^{2k}\Big] & \le\mathbb{E}_{\lambda}\Big[\mathfrak{M}_{1}(E)|I_{\epsilon}|^{2k}\mathrm{e}_{T}\Big(\sup_{t\in[0,T]}|\xi^{\epsilon}(t)|^{2k}\mathrm{e}_{t}^{-1}\Big)\Big]\\
 & \le C_{p}\mathfrak{M}_{1}(E)|I_{\epsilon}|^{2k}\mathbb{E}_{\lambda}\Big(\sup_{t\in[0,T]}|\xi^{\epsilon}(t)|^{2pk}\mathrm{e}_{t}^{-1}\Big)^{1/p}\\
 & \le\mathfrak{M}_{1}(E)\mathrm{o}(|I_{\epsilon}|^{2k}),
\end{aligned}
\label{eq:-51}
\end{equation}
Moreover, for any $p>1$, 
\[
\begin{aligned}\mathbb{E}_{\lambda}\Big[\Big(\int_{0}^{T}|\xi^{\epsilon}(t)|^{2}\,dt\Big)^{2k}\Big] & \le\mathbb{E}_{\lambda}\Big[\mathrm{e}_{T}\Big(\sup_{t\in[0,T]}|\xi^{\epsilon}(t)|^{4k}\mathrm{e}_{t}^{-1}\Big)\Big]\\
 & \le C_{p}\mathbb{E}_{\lambda}\Big(\sup_{t\in[0,T]}|\xi^{\epsilon}(t)|^{4pk}\mathrm{e}_{t}^{-1}\Big)^{1/p}\\
 & \le\mathrm{O}\big(m_{2pk,\lambda}(I_{\epsilon};E)^{1/p}\big),
\end{aligned}
\]
which implies that 
\[
\mathbb{E}_{\lambda}\Big[\Big(\int_{0}^{T}|\xi^{\epsilon}(t)|^{2}\,dt\Big)^{2k}\Big]\le\mathrm{O}\big(m_{2k,\lambda}(I_{\epsilon};E)\big).
\]
Therefore, 
\begin{equation}
\mathbb{E}_{\lambda}\Big[\Big(\int_{0}^{T}\chi_{1}(t)\,dt\Big)^{2k}\Big]\le\mathfrak{M}_{1}(E)\mathrm{O}(|I_{\epsilon}|^{2k})+\mathrm{O}\big(m_{2k,\lambda}(I_{\epsilon};E)\big).\label{eq:-45}
\end{equation}

Next, we estimate $\mathbb{E}_{\lambda}\big[\big(\int_{0}^{T}\chi_{1}(t)\,d\langle W\rangle_{t}\big)^{2k}\big]$.
For any $p,q>1$, by (\ref{eq:-13}), 
\[
\begin{aligned} & \;\quad\mathbb{E}_{\lambda}\Big[\Big(\int_{0}^{T}1_{E_{\epsilon}}(t)|\xi^{\epsilon}(t)|\,d\langle W\rangle_{t}\Big)^{2k}\Big]\\
 & \le\mathbb{E}_{\lambda}\Big[\mathrm{e}_{T}^{2k}\Big(\int_{0}^{T}1_{E_{\epsilon}}(t)|\xi^{\epsilon}(t)|\mathrm{e}_{t}^{-1}\,d\langle W\rangle_{t}\Big)^{2k}\Big]\\
 & \le C_{p}\mathbb{E}_{\lambda}\Big[\Big(\int_{0}^{T}1_{E_{\epsilon}}(t)|\xi^{\epsilon}(t)|\mathrm{e}_{t}^{-1}\,d\langle W\rangle_{t}\Big)^{2pk}\Big]^{1/p}\\
 & \le C_{p}\mathbb{E}_{\lambda}\Big[\Big(\int_{0}^{T}1_{E_{\epsilon}}(t)\,d\langle W\rangle_{t}\Big)^{2pk}\,\Big(\sup_{t\in[0,T]}|\xi^{\epsilon}(t)|\mathrm{e}_{t}^{-1}\Big)^{2pk}\Big]^{1/p}\\
 & \le C_{p}\mathbb{E}_{\lambda}\Big[\Big(\int_{0}^{T}1_{E_{\epsilon}}(t)\,d\langle W\rangle_{t}\Big)^{2pqk}\Big]^{1/(pq)}\mathbb{E}_{\lambda}\Big[\Big(\sup_{t\in[0,T]}|\xi^{\epsilon}(t)|\mathrm{e}_{t}^{-1}\Big)^{2pq^{\prime}k}\Big]^{1/(pq^{\prime})}\\
 & \le C_{p}m_{2pqk,\lambda}(I_{\epsilon};E)^{1/(pq)}\mathfrak{T}_{4pq^{\prime}k}(\xi^{\epsilon})^{1/(pq^{\prime})}
\end{aligned}
\]
which, in view of the fact that $\mathfrak{T}_{4pq^{\prime}k}(\xi^{\epsilon})=\mathrm{o}(1)$,
implies that 
\begin{equation}
\mathbb{E}_{\lambda}\Big[\Big(\int_{0}^{T}1_{E_{\epsilon}}(t)|\xi^{\epsilon}(t)|\,d\langle W\rangle_{t}\Big)^{2k}\Big]\le\mathrm{o}\big(m_{2k,\lambda}(I_{\epsilon};E)\big).\label{eq:-52}
\end{equation}
Moreover, by Lemma \ref{lem:-3} again, 
\[
\begin{aligned} & \;\quad\mathbb{E}_{\lambda}\Big[\Big(\int_{0}^{T}|\xi^{\epsilon}(t)|^{2}\,d\langle W\rangle_{t}\Big)^{2k}\Big]\\
 & \le\mathbb{E}_{\lambda}\Big[\mathrm{e}_{T}^{4k}\Big(\int_{0}^{T}|\xi^{\epsilon}(t)|\mathrm{e}^{-1}\,d\langle W\rangle_{t}\Big)^{2k}\Big(\sup_{t\in[0,T]}|\xi^{\epsilon}(t)|\mathrm{e}_{t}^{-1}\Big)^{2k}\Big]\\
 & \le C_{p}\mathbb{E}_{\lambda}\Big[\Big(\int_{0}^{T}|\xi^{\epsilon}(t)|\mathrm{e}^{-1}\,d\langle W\rangle_{t}\Big)^{2pk}\Big(\sup_{t\in[0,T]}|\xi^{\epsilon}(t)|\mathrm{e}_{t}^{-1}\Big)^{2pk}\Big]^{1/p}\\
 & \le C_{p}\mathfrak{T}_{4pk}(\xi^{\epsilon})^{1/p},
\end{aligned}
\]
which, by (\ref{eq:-13}), implies that 
\[
\mathbb{E}_{\lambda}\Big[\Big(\int_{0}^{T}|\xi^{\epsilon}(t)|^{2}\,d\langle W\rangle_{t}\Big)^{2k}\Big]\le\mathfrak{M}_{1}(E)\mathrm{O}\big(|I_{\epsilon}|^{2k})+\mathrm{O}\big(m_{2k,\lambda}(I_{\epsilon};E)\big).
\]
Therefore, 
\begin{equation}
\mathbb{E}_{\lambda}\Big[\Big(\int_{0}^{T}\chi_{1}(t)\,d\langle W\rangle_{t}\Big)^{2k}\Big]\le\mathfrak{M}_{1}(E)\mathrm{O}(|I_{\epsilon}|^{2k})+\mathrm{O}\big(m_{2k,\lambda}(I_{\epsilon};E)\big).\label{eq:-48}
\end{equation}

We now estimate $\mathbb{E}_{\lambda}\big[\big(\int_{0}^{T}\chi_{1}(t)^{2}\,d\langle W\rangle_{t}\big)^{k}\big]$.
Similarly to the above, for any $p>1$, 
\[
\begin{aligned} & \;\quad\mathbb{E}_{\lambda}\Big[\Big(\int_{0}^{T}1_{E_{\epsilon}}(t)|\xi^{\epsilon}(t)|^{2}\,d\langle W\rangle_{t}\Big)^{k}\Big]\\
 & \le\mathbb{E}_{\lambda}\Big[\mathrm{e}_{T}^{k}\Big(\int_{0}^{T}1_{E_{\epsilon}}(t)|\xi^{\epsilon}(t)|^{2}\mathrm{e}_{t}^{-1}\,d\langle W\rangle_{t}\Big)^{k}\Big]\\
 & \le C_{p}\mathbb{E}_{\lambda}\Big[\Big(\int_{0}^{T}1_{E_{\epsilon}}(t)|\xi^{\epsilon}(t)|^{2}\mathrm{e}_{t}^{-1}\,d\langle W\rangle_{t}\Big)^{pk}\Big]^{1/p}\\
 & \le C_{p}\mathbb{E}_{\lambda}\Big[\Big(\int_{0}^{T}1_{E_{\epsilon}}(t)\,d\langle W\rangle_{t}\Big)^{pk}\,\Big(\sup_{t\in[0,T]}|\xi^{\epsilon}(t)|^{2}\mathrm{e}_{t}^{-1}\Big)^{pk}\Big]^{1/p}\\
 & \le C_{p}\mathbb{E}_{\lambda}\Big[\Big(\int_{0}^{T}1_{E_{\epsilon}}(t)\,d\langle W\rangle_{t}\Big)^{2pk}+\Big(\sup_{t\in[0,T]}|\xi^{\epsilon}(t)|^{2}\mathrm{e}_{t}^{-1}\Big)^{2pk}\Big]^{1/p}\\
 & \le C_{p}[m_{2pk,\lambda}(I_{\epsilon};E)+\mathfrak{T}_{4pk}(\xi^{\epsilon})]^{1/p}\\
 & \le C_{p}\big[\mathfrak{M}_{1}(E)\mathrm{O}(|I_{\epsilon}|^{2k})+m_{2pk,\lambda}(I_{\epsilon};E)^{1/p}\big],
\end{aligned}
\]
which implies that 
\[
\mathbb{E}_{\lambda}\Big[\Big(\int_{0}^{T}1_{E_{\epsilon}}(t)\mathrm{O}(|\xi^{\epsilon}(t)|^{2})\,d\langle W\rangle_{t}\Big)^{k}\Big]\le\mathfrak{M}_{1}(E)\mathrm{O}(|I_{\epsilon}|^{2k})+\mathrm{O}\big(m_{2k,\lambda}(I_{\epsilon};E)\big).
\]
Moreover, for any $p>1$, by Young's inequality, 
\[
\begin{aligned} & \;\quad\mathbb{E}_{\lambda}\Big[\Big(\int_{0}^{T}|\xi^{\epsilon}(t)|^{4}\,d\langle W\rangle_{t}\Big)^{k}\Big]\\
 & \le\mathbb{E}_{\lambda}\Big[\mathrm{e}_{T}^{2k}\Big(\int_{0}^{T}|\xi^{\epsilon}(t)|^{4}\mathrm{e}_{t}^{-2}\,d\langle W\rangle_{t}\Big)^{k}\Big]\\
 & \le C_{p}\mathbb{E}_{\lambda}\Big[\Big(\int_{0}^{T}|\xi^{\epsilon}(t)|\mathrm{e}_{t}^{-1}\,d\langle W\rangle_{t}\Big)^{pk}\Big(\sup_{t\in[0,T]}|\xi^{\epsilon}(t)|^{3}\mathrm{e}_{t}^{-1}\Big)^{pk}\Big]^{1/p}\\
 & \le C_{p}\mathbb{E}_{\lambda}\Big[\Big(\int_{0}^{T}|\xi^{\epsilon}(t)|\mathrm{e}_{t}^{-1}\,d\langle W\rangle_{t}\Big)^{4pk}+\Big(\sup_{t\in[0,T]}|\xi^{\epsilon}(t)|^{3}\mathrm{e}_{t}^{-1}\Big)^{4pk/3}\Big]^{1/p}\\
 & \le C_{p}\mathfrak{T}_{4pk}(\xi^{\epsilon})^{1/p},
\end{aligned}
\]
which, together with (\ref{eq:-13}), implies that 
\[
\mathbb{E}_{\lambda}\Big[\Big(\int_{0}^{T}|\xi^{\epsilon}(t)|^{4}\,d\langle W\rangle_{t}\Big)^{k}\Big]\le\mathfrak{M}_{1}(E)\mathrm{O}(|I_{\epsilon}|^{2k})+\mathrm{O}\big(m_{2k,\lambda}(I_{\epsilon};E)\big).
\]
Hence, 
\begin{equation}
\mathbb{E}_{\lambda}\Big[\Big(\int_{0}^{T}\chi_{1}(t)^{2}\,d\langle W\rangle_{t}\Big)^{k}\Big]\le\mathfrak{M}_{1}(E)\mathrm{O}(|I_{\epsilon}|^{2k})+\mathrm{O}\big(m_{2k,\lambda}(I_{\epsilon};E)\big).\label{eq:-49}
\end{equation}

With the estimates (\ref{eq:-45})\textendash (\ref{eq:-49}), we
are now in a position to apply Lemma \ref{lem:} and deduce (\ref{eq:-11}).
The proof of (\ref{eq:-15}) is similar to that of (\ref{eq:-11}),
except that in the derivation, we need to use both (\ref{eq:-13}),
(\ref{eq:-14}), and (\ref{eq:-11}). 

The proof of (\ref{eq:-12}) is also similar in essence to that of
(\ref{eq:-11}). By a second order Taylor expansion, it is not difficult
to see that, for $\varphi=b_{1},b_{2},\sigma$, 
\[
\begin{aligned} & \;\quad\varphi(t,x^{\epsilon}(t),u^{\epsilon}(t))-\varphi(t)\\
 & =\partial_{x}\varphi(t)\xi^{\epsilon}+\frac{1}{2}\partial_{x}^{2}\varphi(t)(\xi^{\epsilon})^{2}+\delta\varphi(t)+\delta(\partial_{x}\varphi)(t)\xi^{\epsilon}+\delta(\partial_{x}^{2}\varphi)(t)(\xi^{\epsilon})^{2}+\mathrm{O}(|\xi^{\epsilon}|^{3})\\
 & =\partial_{x}\varphi(t)\xi^{\epsilon}+\frac{1}{2}\partial_{x}^{2}\varphi(t)(\xi^{\epsilon})^{2}+\delta\varphi(t)+\delta(\partial_{x}\varphi)(t)\xi^{\epsilon}+1_{E_{\epsilon}}(t)\mathrm{O}(|\xi^{\epsilon}|^{2})+\mathrm{O}(|\xi^{\epsilon}|^{3})\\
 & =\partial_{x}\varphi(t)\xi^{\epsilon}+\frac{1}{2}\partial_{x}^{2}\varphi(t)(y^{\epsilon})^{2}+\delta\varphi(t)+\delta(\partial_{x}\varphi)(t)y^{\epsilon}+1_{E_{\epsilon}}(t)\mathrm{O}(|\xi^{\epsilon}|^{2})+\mathrm{O}(|\xi^{\epsilon}|^{3})\\
 & \quad+1_{E_{\epsilon}}(t)\mathrm{O}(|\xi^{\epsilon}-y^{\epsilon}|)+\mathrm{O}(|\xi^{\epsilon}-y^{\epsilon}||y^{\epsilon}|)+\mathrm{O}(|\xi^{\epsilon}-y^{\epsilon}|^{2}).
\end{aligned}
\]
Therefore, 
\begin{equation}
\begin{aligned} & \;\quad\varphi(t,x^{\epsilon}(t),u^{\epsilon}(t))-\varphi(t)\\
 & =\partial_{x}\varphi(t)\xi^{\epsilon}+\frac{1}{2}\partial_{x}^{2}\varphi(t)(y^{\epsilon})^{2}+\delta\varphi(t)+1_{E_{\epsilon}}(t)\mathrm{O}(|\xi^{\epsilon}|)+\mathrm{O}(|\xi^{\epsilon}|^{3})\\
 & \quad+1_{E_{\epsilon}}(t)\mathrm{O}(|\xi^{\epsilon}-y^{\epsilon}|)+\mathrm{O}(|\xi^{\epsilon}-y^{\epsilon}||y^{\epsilon}|)+\mathrm{O}(|\xi^{\epsilon}-y^{\epsilon}|^{2}),
\end{aligned}
\label{eq:-233}
\end{equation}
for $\varphi=b_{1},b_{2}$, and 
\begin{equation}
\begin{aligned} & \;\quad\sigma(t,x^{\epsilon}(t),u^{\epsilon}(t))-\sigma(t)\\
 & =\partial_{x}\sigma(t)\xi^{\epsilon}+\frac{1}{2}\partial_{x}^{2}\sigma(t)(y^{\epsilon})^{2}+\delta\sigma(t)+\delta(\partial_{x}\sigma)(t)y^{\epsilon}+1_{E_{\epsilon}}(t)\mathrm{O}(|\xi^{\epsilon}|^{2})+\mathrm{O}(|\xi^{\epsilon}|^{3})\\
 & \quad+1_{E_{\epsilon}}(t)\mathrm{O}(|\xi^{\epsilon}-y^{\epsilon}|)+\mathrm{O}(|\xi^{\epsilon}-y^{\epsilon}||y^{\epsilon}|)+\mathrm{O}(|\xi^{\epsilon}-y^{\epsilon}|^{2}).
\end{aligned}
\label{eq:-234}
\end{equation}
Let $\zeta^{\epsilon}=\xi^{\epsilon}-y^{\epsilon}-z^{\epsilon}$,
and 
\[
\chi_{2}(t)=1_{E_{\epsilon}}(t)\mathrm{O}(|\xi^{\epsilon}|^{2})+\mathrm{O}(|\xi^{\epsilon}|^{3})+1_{E_{\epsilon}}(t)\mathrm{O}(|\xi^{\epsilon}-y^{\epsilon}|)+\mathrm{O}(|\xi^{\epsilon}-y^{\epsilon}||y^{\epsilon}|)+\mathrm{O}(|\xi^{\epsilon}-y^{\epsilon}|^{2}).
\]
Then, by substituting (\ref{eq:-233}) and (\ref{eq:-234}) into the
SDE of $\xi^{\epsilon}$, we have 
\[
\begin{aligned}d\zeta^{\epsilon} & =[\partial_{x}b_{1}(t)\zeta^{\epsilon}+1_{E_{\epsilon}}(t)\mathrm{O}(|\xi^{\epsilon}|)+\chi_{2}(t)]dt\\
 & \quad+[\partial_{x}b_{2}(t)\zeta^{\epsilon}+1_{E_{\epsilon}}(t)\mathrm{O}(|\xi^{\epsilon}|)+\chi_{2}(t)]d\langle W\rangle_{t}\\
 & \quad+[\partial_{x}b_{2}(t)\zeta^{\epsilon}+\chi_{2}(t)]dW_{t}.
\end{aligned}
\]
In view of (\ref{eq:-45}) and (\ref{eq:-52}), in order to apply
Lemma \ref{lem:}, it suffices to estimate $\mathbb{E}_{\lambda}\big[\big(\int_{0}^{T}\chi_{2}(t)\,dt\big)^{2k}\big]$,
$\mathbb{E}_{\lambda}\big[\big(\int_{0}^{T}\chi_{2}(t)\,d\langle W\rangle_{t}\big)^{2k}\big]$
and $\mathbb{E}_{\lambda}\big[\big(\int_{0}^{T}\chi_{2}(t)^{2}\,d\langle W\rangle_{t}\big)^{k}\big]$,
which can be done similarly to those of $\chi_{1}(t)$ in the above
using the established estimates (\ref{eq:-13}), (\ref{eq:-14}),
and (\ref{eq:-11}). 
\end{proof}
\begin{proof}[Proof of Theorem \ref{thm:}]
Let $E$ be the progressively measurable set defined by (\ref{eq:-40}).
By definition of $J(\cdot)$, we have 
\[
\begin{aligned} & \;\quad J(u^{\epsilon})-J(\bar{u})\\
 & =\mathbb{E}_{\lambda}\bigg\{\partial_{x}h(\bar{x}(T))\xi^{\epsilon}(T)+\Big(\int_{0}^{1}\theta\partial_{x}^{2}h\big(\bar{x}(T)+\theta\xi^{\epsilon}(T)\big)d\theta\Big)\xi^{\epsilon}(T)^{2}\\
 & \quad+\int_{0}^{T}\Big[\delta f_{1}(t)+\partial_{x}f_{1}\big(t,\bar{x}(t),u^{\epsilon}(t)\big)\xi^{\epsilon}(t)+\Big(\int_{0}^{1}\theta\partial_{x}^{2}f_{1}\big(t,\bar{x}(t)+\theta\xi^{\epsilon}(t),u^{\epsilon}(t)\big)d\theta\Big)\xi^{\epsilon}(t)^{2}\Big]dt\\
 & \quad+\int_{0}^{T}\Big[\delta f_{2}(t)+\partial_{x}f_{2}\big(t,\bar{x}(t),u^{\epsilon}(t)\big)\xi^{\epsilon}(t)+\Big(\int_{0}^{1}\theta\partial_{x}^{2}f_{2}\big(t,\bar{x}(t)+\theta\xi^{\epsilon}(t),u^{\epsilon}(t)\big)d\theta\Big)\xi^{\epsilon}(t)^{2}\Big]d\langle W\rangle_{t}\bigg\}.
\end{aligned}
\]
Notice that we have the following approximations 
\begin{align}
\xi^{\epsilon} & =y^{\epsilon}+z^{\epsilon}+\mathfrak{M}_{1}(E)\,\mathrm{o}(|I_{\epsilon}|)+\mathrm{o}\big(m_{1,\lambda}(I_{\epsilon};E)\big),\label{eq:-240}\\
(\xi^{\epsilon})^{2} & =(y^{\epsilon})^{2}+\mathfrak{M}_{1}(E)\,\mathrm{o}(|I_{\epsilon}|)+\mathrm{o}\big(m_{1,\lambda}(I_{\epsilon};E)\big),\label{eq:-241}\\
\Big(\int_{0}^{1} & \theta\varphi\big(t,\bar{x}(t)+\theta\xi^{\epsilon}(t),u^{\epsilon}(t)\big)d\theta\Big)(\xi^{\epsilon})^{2}=\frac{1}{2}\varphi(t)(y^{\epsilon})^{2}+\mathfrak{M}_{1}(E)\,\mathrm{o}(|I_{\epsilon}|)\nonumber \\
 & \hspace{20mm}+\mathrm{o}\big(m_{1,\lambda}(I_{\epsilon};E)\big),\;\;\text{for}\;\varphi=\partial_{x}^{2}h,\partial_{x}^{2}f_{1},\partial_{x}^{2}f_{2}.\label{eq:-242}
\end{align}
The approximation (\ref{eq:-240}) follows directly from (\ref{eq:-12}).
The approximation (\ref{eq:-241}) follows from $|(\xi^{\epsilon})^{2}-(y^{\epsilon})^{2}|\le(|\xi|^{\epsilon}+|y^{\epsilon}|)|\xi^{\epsilon}-y^{\epsilon}|$
together with (\ref{eq:-13}), (\ref{eq:-14}), and (\ref{eq:-11})
in Lemma \ref{lem:-2}. For (\ref{eq:-242}), in view of $\mathrm{supp}(\delta\varphi)\subseteq E_{\epsilon}=E\cap(I_{\epsilon}\times\Omega)$
and the boundedness of $\partial_{x}\varphi$, we have 
\[
\begin{aligned} & \;\quad\Big(\int_{0}^{1}\theta\varphi\big(t,\bar{x}(t)+\theta\xi^{\epsilon}(t),u^{\epsilon}(t)\big)d\theta\Big)(\xi^{\epsilon})^{2}\\
 & =\Big(\int_{0}^{1}\theta\varphi\big(t,\bar{x}(t),u^{\epsilon}(t)\big)d\theta\Big)(\xi^{\epsilon})^{2}+\mathrm{O}(|\xi^{\epsilon}|^{3})\\
 & =\Big(\int_{0}^{1}\theta\varphi(t)d\theta\Big)(\xi^{\epsilon})^{2}+\frac{1}{2}\delta\varphi(\xi^{\epsilon})^{2}+\mathrm{O}(|\xi^{\epsilon}|^{3})\\
 & =\frac{1}{2}\varphi(t)(\xi^{\epsilon})^{2}+\mathrm{O}(1_{E_{\epsilon}}|\xi^{\epsilon}|^{2})+\mathrm{O}(|\xi^{\epsilon}|^{3}).
\end{aligned}
\]
By $\int_{0}^{T}1_{E_{\epsilon}}(t)|\xi^{\epsilon}(t)|^{2}dt\le|I_{\epsilon}|\mathrm{e}_{T}\sup_{t\in[0,T]}\xi^{\epsilon}(t)\mathrm{e}_{t}^{-1}$
and Lemma \ref{lem:-2}, it is easily seen that 
\[
\mathbb{E}_{\lambda}\big(\int_{0}^{T}1_{E_{\epsilon}}(t)|\xi^{\epsilon}(t)|^{2}dt+\sup_{t\in[0,T]}|\xi^{\epsilon}(t)|^{3}\big)=\mathfrak{M}_{1}(E)\,\mathrm{o}(|I_{\epsilon}|)+\mathrm{o}\big(m_{1,\lambda}(I_{\epsilon};E)\big),
\]
which yields the approximation (\ref{eq:-242}). Therefore, 
\begin{equation}
\begin{aligned} & \;\quad J(u^{\epsilon})-J(\bar{u})\\
 & =\mathbb{E}_{\lambda}\bigg\{\partial_{x}h(\bar{x}(T))(y^{\epsilon}(T)+z^{\epsilon}(T))+\frac{1}{2}\partial_{x}^{2}h\big(\bar{x}(T)\big)y^{\epsilon}(T)^{2}\\
 & \quad+\int_{0}^{T}\Big[\delta f_{1}(t)+\partial_{x}f_{1}(t)(y^{\epsilon}(t)+z^{\epsilon}(t))+\frac{1}{2}\partial_{x}^{2}f_{1}(t)y^{\epsilon}(t)^{2}\Big]dt\\
 & \quad+\int_{0}^{T}\Big[\delta f_{2}(t)+\partial_{x}f_{2}(t)(y^{\epsilon}(t)+z^{\epsilon}(t))+\frac{1}{2}\partial_{x}^{2}f_{2}(t)y^{\epsilon}(t)^{2}\Big]d\langle W\rangle_{t}\bigg\}\\
 & \quad+\mathfrak{M}_{1}(E)\,\mathrm{o}(|I_{\epsilon}|)+\mathrm{o}\big(m_{1,\lambda}(I_{\epsilon};E)\big),
\end{aligned}
\label{eq:-243}
\end{equation}

Next, we transform the cost $\mathbb{E}_{\lambda}[\partial_{x}h(\bar{x}(T))(y^{\epsilon}(T)+z^{\epsilon}(T))]$
into a cumulative one. By (\ref{eq:-16}), 
\begin{equation}
\begin{aligned} & \;\quad\mathbb{E}_{\lambda}\big[-\partial_{x}h(\bar{x}(T))(y^{\epsilon}(T)+z^{\epsilon}(T))\big]\\
 & =\mathbb{E}_{\lambda}\big[p(T)(y^{\epsilon}(T)+z^{\epsilon}(T))\big]+\mathfrak{M}_{1}(E)\,\mathrm{o}(|I_{\epsilon}|)+\mathrm{o}\big(m_{1,\lambda}(I_{\epsilon};E)\big)\\
 & =\mathbb{E}_{\lambda}\Big\{\int_{0}^{T}\Big(\delta b_{1}(t)p(t)+\partial_{x}f_{1}(t)\big(y^{\epsilon}(t)+z^{\epsilon}(t)\big)+\frac{1}{2}\partial_{x}^{2}b_{1}(t)p(t)y^{\epsilon}(t)^{2}\Big)dt\\
 & \quad+\int_{0}^{T}\Big(\delta b_{2}(t)p(t)+\delta\sigma(t)q(t)+\partial_{x}f_{2}(t)\big(y^{\epsilon}(t)+z^{\epsilon}(t)\big)\\
 & \quad+\frac{1}{2}\big[\partial_{x}^{2}b_{2}(t)p(t)y^{\epsilon}(t)^{2}+\partial_{x}^{2}\sigma(t)q(t)\big]y^{\epsilon}(t)^{2}+\delta(\partial_{x}\sigma)(t)q(t)y^{\epsilon}(t)\Big)\,d\langle W\rangle_{t}\Big\}\\
 & \quad+\mathfrak{M}_{1}(E)\,\mathrm{o}(|I_{\epsilon}|)+\mathrm{o}\big(m_{1,\lambda}(I_{\epsilon};E)\big).
\end{aligned}
\label{eq:-239}
\end{equation}
Notice that the last integral term $\mathbb{E}_{\lambda}\big(\int_{0}^{T}\delta(\partial_{x}\sigma)(t)q(t)y^{\epsilon}(t)d\langle W\rangle_{t}\big)$
in the above is also of order $\mathfrak{M}_{1}(E)\,\mathrm{o}(|I_{\epsilon}|)+\mathrm{o}\big(m_{1,\lambda}(I_{\epsilon};E)\big)$.
To see this, by \citep[Theorem 3.5]{LQ16}, $\mathbb{E}_{\lambda}\big(\int_{0}^{T}q(t)^{2}\mathrm{e}_{t}d\langle W\rangle_{t}\big)$
is bounded. Therefore, for any $k\ge2$, in view of $\mathrm{supp}(\delta(\partial_{x}\sigma))\subseteq E_{\epsilon}$
and Lemma \ref{lem:-3}, 
\[
\begin{aligned} & \;\quad\Big|\mathbb{E}_{\lambda}\Big(\int_{0}^{T}\delta(\partial_{x}\sigma)(t)q(t)y^{\epsilon}(t)d\langle W\rangle_{t}\Big)\Big|\\
 & \le C_{k}\mathbb{E}_{\lambda}\Big(\int_{0}^{T}1_{E_{\epsilon}}(t)y^{\epsilon}(t)^{k}\mathrm{e}_{t}^{-1}d\langle W\rangle_{t}\Big)^{1/k}\\
 & \le C_{k}\mathbb{E}_{\lambda}\Big[\Big(\int_{0}^{T}1_{E_{\epsilon}}(t)d\langle W\rangle_{t}\Big)\Big(\sup_{t\in[0,T]}y^{\epsilon}(t)^{k}\mathrm{e}_{t}^{-1}\Big)\Big]^{1/k}\\
 & \le C_{k}\mathbb{E}_{\lambda}\Big[\Big(\int_{0}^{T}1_{E_{\epsilon}}(t)d\langle W\rangle_{t}\Big)^{2}\Big]^{1/2}\mathbb{E}_{\lambda}\Big[\Big(\sup_{t\in[0,T]}y^{\epsilon}(t)^{2k}\mathrm{e}_{t}^{-1}\Big)\Big]^{1/(2k)}\\
 & \le C_{k}m_{2,\lambda}(I_{\epsilon};E)^{1/2}[\mathfrak{M}_{1}(E)\,\mathrm{o}(|I_{\epsilon}|^{1/2})+\mathrm{o}\big(m_{k,\lambda}(I_{\epsilon};E)^{1/(2k)}\big)]\\
 & \le\mathfrak{M}_{1}(E)\,\mathrm{o}(|I_{\epsilon}|)+\mathrm{o}\big(m_{2,\lambda}(I_{\epsilon};E)^{1/2}\big)+\mathrm{o}\big(m_{k,\lambda}(I_{\epsilon};E)^{1/k}\big)\\
 & =\mathfrak{M}_{1}(E)\,\mathrm{o}(|I_{\epsilon}|)+\mathrm{o}\big(m_{1,\lambda}(I_{\epsilon};E)\big).
\end{aligned}
\]
Hence, the equality (\ref{eq:-239}) can be further written as 
\begin{equation}
\begin{aligned} & \;\quad\mathbb{E}_{\lambda}\big[-\partial_{x}h(\bar{x}(T))(y^{\epsilon}(T)+z^{\epsilon}(T))\big]\\
 & =\mathbb{E}_{\lambda}\bigg\{\int_{0}^{T}\Big[\delta b_{1}(t)p(t)+\partial_{x}f_{1}(t)\big(y^{\epsilon}(t)+z^{\epsilon}(t)\big)+\frac{1}{2}\partial_{x}^{2}b_{1}(t)p(t)y^{\epsilon}(t)^{2}\Big]dt\\
 & \quad+\int_{0}^{T}\Big[\delta b_{2}(t)p(t)+\delta\sigma(t)q(t)+\partial_{x}f_{2}(t)\big(y^{\epsilon}(t)+z^{\epsilon}(t)\big)\\
 & \quad+\frac{1}{2}\big[\partial_{x}^{2}b_{2}(t)p(t)y^{\epsilon}(t)^{2}+\partial_{x}^{2}\sigma(t)q(t)\big]y^{\epsilon}(t)^{2}\Big]\,d\langle W\rangle_{t}\bigg\}\\
 & \quad+\mathfrak{M}_{1}(E)\,\mathrm{o}(|I_{\epsilon}|)+\mathrm{o}\big(m_{1,\lambda}(I_{\epsilon};E)\big).
\end{aligned}
\label{eq:-244}
\end{equation}

Also, we transform $\mathbb{E}_{\lambda}[\partial_{x}^{2}h\big(\bar{x}(T)\big)y^{\epsilon}(T)^{2}]$
into a cumulative cost. By (\ref{eq:-17}), 
\begin{align*}
 & \;\quad\mathbb{E}_{\lambda}\big[-\partial_{x}^{2}h\big(\bar{x}(T)\big)y^{\epsilon}(T)^{2}\big]=\mathbb{E}_{\lambda}\big[P(T)y^{\epsilon}(T)^{2}\big]\\
 & =\mathbb{E}_{\lambda}\Big\{\int_{0}^{T}\big[\partial_{x}^{2}f_{1}(t)-\partial_{x}^{2}b_{1}(t)p(t)\big]y^{\epsilon}(t)^{2}dt\\
 & \quad+\int_{0}^{T}\Big(\big[\partial_{x}^{2}f_{2}(t)-\partial_{x}^{2}b_{2}(t)p(t)-\partial_{x}^{2}\sigma(t)q(t)\big]y^{\epsilon}(t)^{2}+\delta\sigma(t)^{2}P(t)\\
 & \quad+\big[2\partial_{x}\sigma(t)P(t)+Q(t)\big]\delta\sigma(t)y^{\epsilon}(t)\Big)\,d\langle W\rangle_{t}\Big\}.
\end{align*}
Similar to before, it can be shown that the term $\mathbb{E}_{\lambda}\big(\int_{0}^{T}[2\partial_{x}\sigma(t)P(t)+Q(t)]\delta\sigma(t)y^{\epsilon}(t)\,d\langle W\rangle_{t}\big)$
is of order $\mathfrak{M}_{1}(E)\,\mathrm{o}(|I_{\epsilon}|)+\mathrm{o}\big(m_{1,\lambda}(I_{\epsilon};E)\big)$.
Therefore, 
\begin{equation}
\begin{aligned} & \;\quad\mathbb{E}_{\lambda}\big[-\partial_{x}^{2}h\big(\bar{x}(T)\big)y^{\epsilon}(T)^{2}\big]\\
 & =\mathbb{E}_{\lambda}\Big\{\int_{0}^{T}\big[\partial_{x}^{2}f_{1}(t)-\partial_{x}^{2}b_{1}(t)p(t)\big]y^{\epsilon}(t)^{2}dt\\
 & \quad+\int_{0}^{T}\Big[\Big(\partial_{x}^{2}f_{2}(t)-\partial_{x}^{2}b_{2}(t)p(t)-\partial_{x}^{2}\sigma(t)q(t)\Big)y^{\epsilon}(t)^{2}+\delta\sigma(t)^{2}P(t)\Big]\,d\langle W\rangle_{t}\Big\}\\
 & \quad+\mathfrak{M}_{1}(E)\,\mathrm{o}(|I_{\epsilon}|)+\mathrm{o}\big(m_{1,\lambda}(I_{\epsilon};E)\big).
\end{aligned}
\label{eq:-245}
\end{equation}

Combining (\ref{eq:-243}), (\ref{eq:-244}), and (\ref{eq:-245}),
we arrive at 
\begin{equation}
\begin{aligned} & \;\quad J(u^{\epsilon})-J(\bar{u})\\
 & =\mathbb{E}_{\lambda}\Big[\int_{0}^{T}\Big(\delta f_{1}(t)-\delta b_{1}(t)p(t)\Big)dt\Big]\\
 & +\mathbb{E}_{\lambda}\Big[\int_{0}^{T}\Big(\delta f_{2}(t)-\delta b_{2}(t)p(t)-\delta\sigma(t)q(t)-\frac{1}{2}\delta\sigma(t)^{2}P(t)\Big)d\langle W\rangle_{t}\Big]\\
 & +\mathfrak{M}_{1}(E)\,\mathrm{o}(|I_{\epsilon}|)+\mathrm{o}\big(m_{1,\lambda}(I_{\epsilon};E)\big).
\end{aligned}
\label{eq:-41}
\end{equation}
We now show that the optimality of $\bar{u}$ and (\ref{eq:-41})
implies that
\begin{equation}
\left\{ \begin{array}{c}
\delta f_{1}(t)-\delta b_{1}(t)p(t)\ge0,\hspace{20mm}\hspace{20mm}\mathfrak{M}_{1}\text{-a.e.},\\
\delta f_{2}(t)-\delta b_{2}(t)p(t)-\delta\sigma(t)q(t)-\frac{1}{2}\delta\sigma(t)^{2}P(t)\ge0,\;\;\mathfrak{M}_{2}\text{-a.e.}
\end{array}\right.\label{eq:-42}
\end{equation}
By separability of $\mathbb{U}$ and the continuity of $H_{1},H_{2}$
in $u$, there exist progressively measurable processes $\bar{u}_{1}$
and $\bar{u}_{2}$ such that 
\[
\begin{aligned}H_{1}(t,\bar{x}(t),\bar{u}_{1}(t)) & =\max_{u\in\mathbb{U}}H_{1}(t,\bar{x}(t),u),\;\;\mathfrak{M}_{1}\text{-a.e.},\\
H_{2}(t,\bar{x}(t),\bar{u}_{2}(t)) & =\max_{u\in\mathbb{U}}H_{2}(t,\bar{x}(t),u),\;\;\mathfrak{M}_{2}\text{-a.e.}
\end{aligned}
\]
We first set $u_{1}=\bar{u}_{1},\,u_{2}=\bar{u}$. Then $\mathfrak{M}_{2}(E)=0$,
and therefore $m_{2,\lambda}(I_{\epsilon};E)=0$. Moreover, (\ref{eq:-41})
reduces to 
\[
J(u^{\epsilon})-J(\bar{u})=\mathbb{E}_{\lambda}\Big[\int_{0}^{T}\Big(\delta f_{1}(t)-\delta b_{1}(t)p(t)\Big)dt\Big]+\mathfrak{M}_{1}(E)\,\mathrm{o}(|I_{\epsilon}|),
\]
which clearly implies the first inequality in (\ref{eq:-42}). 

We now turn to the proof of the second inequality in (\ref{eq:-42}).
For any $a>0$, let 
\[
E_{a}=\{(t,\omega):H_{2}(t,\bar{x}(t),\bar{u}_{2}(t))-H_{2}(t,\bar{x}(t),\bar{u}(t))\ge a\}.
\]
Set $u_{1}=\bar{u},\,u_{2}=\bar{u}_{2}1_{E_{a}}+\bar{u}1_{E_{a}^{c}}$.
Then $E=E_{a}$ and $\mathfrak{M}_{1}(E)=0$. Therefore, (\ref{eq:-41})
reduces to 
\[
\begin{aligned}J(u^{\epsilon})-J(\bar{u}) & =\mathbb{E}_{\lambda}\Big[\int_{0}^{T}\Big(\delta f_{2}(t)-\delta b_{2}(t)p(t)-\delta\sigma(t)q(t)-\frac{1}{2}\delta\sigma(t)^{2}P(t)\Big)d\langle W\rangle_{t}\Big]\\
 & \quad+\mathrm{o}\big(m_{1,\lambda}(I_{\epsilon};E_{a})\big).
\end{aligned}
\]
By the definition of $E_{a}$ and $u_{2}$, we have 
\[
\delta f_{2}(t)-\delta b_{2}(t)p(t)-\delta\sigma(t)q(t)-\frac{1}{2}\delta\sigma(t)^{2}P(t)\le-a,\;\;\text{on}\;E_{a}.
\]
Therefore, 
\[
0\le J(u^{\epsilon})-J(\bar{u})\le-a\,m_{1,\lambda}(I_{\epsilon};E_{a})+\mathrm{o}\big(m_{1,\lambda}(I_{\epsilon};E_{a})\big),
\]
which clearly implies 
\[
\mathbb{E}_{\lambda}\Big(\int_{I_{\epsilon}}1_{E_{a}}(t,\omega)d\langle W\rangle_{t}\Big)=0,\;\;\text{for all}\;I_{\epsilon}\;\text{with}\;\epsilon\;\text{sufficiently small}.
\]
Therefore, $\mathfrak{M}_{2}(E_{a})=\mathbb{E}_{\lambda}\big(\int_{0}^{T}1_{E_{a}}(t,\omega)d\langle W\rangle_{t}\big)=0$
in view of the arbitrariness of $\{I_{\epsilon}\}_{\epsilon>0}$.
This completes the proof. 
\end{proof}

\section{\label{sec:}An example: linear regulator problem}

Let $\lambda\in\mathcal{P}(\mathbb{S})$ with $\lambda\ll\nu$ and
$a>0$, and take as the decision space $\mathbb{U}=\mathbb{R}$. We
consider the following \emph{linear regulator problem}, which has
wide applications in mathematical finance and engineering (see \citep[p. 23]{FR12}
and references therein):
\begin{equation}
\mathop{\mathrm{minimize}}_{u\in\mathcal{A}[0,1]}\;\mathbb{E}_{\lambda}\bigg(\frac{a}{2}\int_{0}^{1}u(t)^{2}dt+x(1)^{2}\bigg),\label{eq:-25}
\end{equation}
with
\begin{equation}
\left\{ \begin{aligned}dx(t) & =u(t)dt+u(t)d\langle W\rangle_{t}+u(t)dW_{t},\quad t\in(0,1],\\
x(0) & =1.
\end{aligned}
\right.\label{eq:-26}
\end{equation}

Suppose that $(\bar{x}(\cdot),\bar{u}(\cdot))$ is an optimal pair
of the problem (\ref{eq:-25}). The adjoint equations are
\begin{equation}
\left\{ \begin{aligned}dp(t) & =q(t)dW_{t},\quad t\in[0,1),\\
p(1) & =-2\bar{x}(1),
\end{aligned}
\right.\label{eq:-27}
\end{equation}
\begin{equation}
\left\{ \begin{aligned}dP(t) & =Q(t)dW_{t},\quad t\in[0,1),\\
P(1) & =-2.
\end{aligned}
\right.\label{eq:-28}
\end{equation}
Clearly, $P(t)=-2,\;Q(t)=0$ is the solution to  (\ref{eq:-28}).

The Hamiltonians are
\[
H_{1}(t,x,u)=up(t)-\frac{a}{2}u^{2},\;H_{2}(t,x,u)=u[p(t)+q(t)]-(u-\bar{u}(t))^{2}.
\]
Let $\mathfrak{M}_{1},\mathfrak{M}_{2}$ be the measures on $[0,\infty)\times\Omega$
given by (\ref{eq:-36}) and (\ref{eq:-37}), and $\mathfrak{M}=\mathfrak{M}_{1}+\mathfrak{M}_{2}$.
By Theorem \ref{thm:},
\[
\bar{u}(t)\frac{d\mathfrak{M}_{1}}{d\mathfrak{M}}=\frac{p(t)}{a}\frac{d\mathfrak{M}_{1}}{d\mathfrak{M}},
\]
and
\[
\bar{u}(t)\frac{d\mathfrak{M}_{2}}{d\mathfrak{M}}=\frac{p(t)+q(t)+2\bar{u}(t)}{2}\frac{d\mathfrak{M}_{2}}{d\mathfrak{M}},
\]
which implies that
\[
q(t)=-p(t),\quad\mathfrak{M}_{2}\text{-a.e.}
\]
It follows from the above and (\ref{eq:-27}) that
\begin{equation}
p(t)=p(0)\exp\Big(-W_{t}-\frac{1}{2}\langle W\rangle_{t}\Big),\quad t\in[0,1].\label{eq:-35}
\end{equation}

Therefore, $(\bar{x}(\cdot),\bar{u}(\cdot),p(0))$ is given by the
system
\begin{equation}
\left\{ \begin{aligned}d\bar{x}(t) & =\frac{p(t)}{a}dt+\bar{u}(t)d\langle W\rangle_{t}+\bar{u}(t)dW_{t},\quad t\in(0,1],\\
\bar{x}(0) & =1,\quad\bar{x}(1)=-\frac{1}{2}p(1),
\end{aligned}
\right.\label{eq:-34}
\end{equation}
where $p(\cdot)$ is given by (\ref{eq:-35}). Note that, compared
to BSDEs, the system (\ref{eq:-34}) takes the random variable $p(0)$
as a part of its solution so that the additional condition $\bar{x}(0)=1$
is satisfied. Therefore, (\ref{eq:-34}) is not a simple SDE or BSDE
but a forward\textendash backward type SDE. 

We now look for a solution to  the form $\bar{x}(t)=\theta(t)p(t)$,
where $\theta(t)$ is a process of the form
\[
\left\{ \begin{aligned}d\theta(t) & =\xi_{1}(t)dt+\xi_{2}(t)d\langle W\rangle_{t}+\eta(t)dW_{t},\quad t\in[0,1),\\
\theta(1) & =-\frac{1}{2}.
\end{aligned}
\right.
\]
By Itô's formula,
\[
d\bar{x}(t)=\xi_{1}(t)p(t)dt+[\xi_{2}(t)-\eta(t)]p(t)d\langle W\rangle_{t}+[\eta(t)-\theta(t)]p(t)dW_{t},\quad t\in(0,1].
\]
Comparing the above with (\ref{eq:-34}) gives that
\[
\xi_{1}(t)p(t)\frac{d\mathfrak{M}_{1}}{d\mathfrak{M}}=\frac{p(t)}{a}\frac{d\mathfrak{M}_{1}}{d\mathfrak{M}},
\]
\[
[\xi_{2}(t)-\eta(t)]p(t)\frac{d\mathfrak{M}_{2}}{d\mathfrak{M}}=\bar{u}(t)\frac{d\mathfrak{M}_{2}}{d\mathfrak{M}}=[\eta(t)-\theta(t)]p(t)\frac{d\mathfrak{M}_{2}}{d\mathfrak{M}}.
\]
Therefore, $\xi_{1}(t)=\frac{1}{a}\;\;\mathfrak{M}_{1}\text{-a.e.}$
and $\xi_{2}(t)=2\eta(t)-\theta(t)\;\;\mathfrak{M}_{2}\text{-a.e.}$
Furthermore, $\theta(t)$ is given by the BSDE
\begin{equation}
\left\{ \begin{aligned}d\theta(t) & =\frac{1}{a}dt+[2\eta(t)-\theta(t)]d\langle W\rangle_{t}+\eta(t)dW_{t},\quad t\in[0,1),\\
\theta(1) & =-\frac{1}{2},
\end{aligned}
\right.\label{eq:-32}
\end{equation}
of which a unique solution $(\theta(\cdot),\eta(\cdot))$ exists (cf.
\citep[Theorem 3.10]{LQ16}).

For the moment, let us assume that $\theta(0)<0$. Then by $\bar{x}(0)=\theta(0)p(0)=1$,
we have that $p(0)=1/\theta(0)$ and
\begin{equation}
p(t)=\frac{1}{\theta(0)}\exp\Big(-W_{t}-\frac{1}{2}\langle W\rangle_{t}\Big).\label{eq:-33}
\end{equation}
The optimal pair $(\bar{x}(\cdot),\bar{u}(\cdot))$ is given by
\begin{equation}
\bar{u}(t)=\frac{p(t)}{a}\frac{d\mathfrak{M}_{1}}{d\mathfrak{M}}+[\eta(t)-\theta(t)]p(t)\frac{d\mathfrak{M}_{2}}{d\mathfrak{M}},\label{eq:-29}
\end{equation}
and
\begin{equation}
\left\{ \begin{aligned}d\bar{x}(t) & =\frac{p(t)}{a}dt+[\eta(t)-\theta(t)]p(t)d\langle W\rangle_{t}+[\eta(t)-\theta(t)]p(t)dW_{t},\quad t\in(0,1],\\
\bar{x}(0) & =1,
\end{aligned}
\right.\label{eq:-30}
\end{equation}
where $(\theta(\cdot),\eta(\cdot))$ and $p(\cdot)$ are given by
(\ref{eq:-32}) and (\ref{eq:-33}).

It remains to show that $\theta(0)<0$. Let
\[
\Phi(t)=\exp\Big(-2W_{t}-\langle W\rangle_{t}\Big),\quad t\in[0,1].
\]
By Itô's formula,
\[
d\Phi(t)=\Phi(t)d\langle W\rangle_{t}-2\Phi(t)dW_{t}.
\]
Therefore,
\[
d[\Phi(t)\theta(t)]=\frac{\Phi(t)}{a}dt+\Phi(t)[\eta(t)-2\theta(t)]dW_{t},
\]
which implies that $\Phi(t)\theta(t)-\frac{1}{a}\int_{0}^{t}\Phi(r)dr$
is a martingale. Therefore,
\[
\begin{aligned}\Phi(t) & \theta(t)-\frac{1}{a}\int_{0}^{t}\Phi(r)dr\\
 & =\mathbb{E}_{\lambda}\bigg(\Phi(1)\theta(1)-\frac{1}{a}\int_{0}^{1}\Phi(r)dr\,\bigg|\,\mathcal{F}_{t}\bigg)\\
 & =-\mathbb{E}_{\lambda}\bigg(\frac{\Phi(1)}{2}+\frac{1}{a}\int_{0}^{1}\Phi(r)dr\,\bigg|\,\mathcal{F}_{t}\bigg),
\end{aligned}
\]
which gives that
\begin{equation}
\theta(t)=-\Phi(t)^{-1}\mathbb{E}_{\lambda}\bigg(\frac{\Phi(1)}{2}+\frac{1}{a}\int_{t}^{1}\Phi(r)dr\,\bigg|\,\mathcal{F}_{t}\bigg).\label{eq:-31}
\end{equation}
This, together with the fact that $\Phi(t)>0$, shows that $\theta(0)<0$.

\section*{Acknowledgements}

The author would like to thank the anonymous referees for providing
many useful comments and suggestions, which helps improve the quality
of the current paper. 

\noindent \renewcommand\bibname{References}

\end{document}